\documentclass{amsart}%{aptpub}
\usepackage{amsmath,amsthm,amssymb}

\usepackage{psfrag}
\usepackage{graphicx}%

\author{O.~Hryniv and M.~Menshikov} 
\address{Department of Mathematical Sciences, Durham University, Science Laboratories, South Rd., Durham DH1~3LE, UK}
\newcommand{\InputPictureHeight}[2]{\includegraphics[height=#1]{#2}}
\newcommand{\proj}[3]{[\,#1\,]_{#2}^{#3}}%
\newcommand{\bs}[1]{\ensuremath{\boldsymbol{#1}}}%
\newcommand{\DF}{\stackrel{\,{\mathsf{def}}}{=}\,}%
\newcommand{\Ind}[1]{\one_{\{#1\}}}%
\newcommand{\one}{\hbox{\rm 1\kern-.27em I}}%
\def\bl{\bigl}%
\def\br{\bigr}%
\def\Bl{\Bigl}%
\def\Br{\Bigr}%
\def\Exp{\ensuremath{\mathsf{Exp}}}%
\def\calA{\mathcal{A}}%
\def\calB{\mathcal{B}}%
\def\calD{\mathcal{D}}%
\def\calJ{\mathcal{J}}%
\def\calK{\mathcal{K}}%
\def\calQ{\mathcal{Q}}%
\def\calS{\mathcal{S}}%
\def\calW{\mathcal{W}}%
\newcommand{\hatcalW}{\widehat{\mathcal{W}}}%
\def\calY{\mathcal{Y}}%
\def\BbbN{\mathbb{N}}%
\def\BbbR{\mathbb{R}}%
\def\BbbZ{\mathbb{Z}}%
\def\bfB{\mathbf{B}}%
\def\frt{\mathfrak{t}}%
\def\fru{\mathfrak{u}}%
\def\sfP{\mathsf{P}}%
\newcommand{\bsw}{\bs{w}}%
\newcommand{\bsy}{\bs{y}}%
\newcommand{\bshatw}{\bs{\hatw}}%
\newcommand{\bshatpi}{\bs{\hatpi}}%
\newcommand{\bspi}{\bs\pi}%
\newcommand{\bspip}{\bs{\pi_+}}%
\newcommand{\bspim}{\bs{\pi_-}}%
\newcommand{\hatpi}{\widehat\pi}%
\newcommand{\hatpmmp}{\hat p^m_{\ominus\oplus}}%
\newcommand{\bstilw}{\bs{\widetilde w}}%
\def\sfE{\mathsf{E}}%
\newcommand{\hatw}{\widehat w}%
\newcommand{\hatx}{\widehat x}%
\newcommand{\haty}{\widehat y}%
\newcommand{\tiltau}{\tilde\tau}%
\newcommand{\tilkappa}{\tilde\kappa}%
\newcommand{\taust}{\tau^*}%
\newcommand{\tilx}{\tilde x}%
\newcommand{\tily}{\tilde y}%
\newcommand{\tilT}{\widetilde{T}}%
\newcommand{\tilX}{\widetilde{X}}%
\newcommand{\barkappa}{\overline\kappa}%
\newcommand{\bartau}{\overline\tau}%
\newcommand{\bars}{\bar s}%
\newcommand{\barz}{\bar z}%
\newcommand{\barcalS}{\overline{\calS}}%
\newcommand{\Top}{T_\oplus}%
\newcommand{\Topop}{T_{\oplus\oplus}}%
\newcommand{\Tom}{T_\ominus}%
\newcommand{\phiop}{\varphi_\oplus}%
\newcommand{\phiom}{\varphi_\ominus}%
\newtheorem{thm}{Theorem}[section]

\newtheorem{lem}[thm]{Lemma}

\newtheorem{cor}[thm]{Corollary}

\newtheorem{rem}{Remark}[thm]%

\newtheorem{prop}[thm]{Proposition}  % Numbered along with thm

\numberwithin{equation}{section}  % If you number theorems, etc. within sections,
\begin{document}

\title{Long time behaviour in a model of microtubule growth} % insert title -

\begin{abstract}
We study a continuous time stochastic process on strings made of two types of particles, whose dynamics mimics the behaviour of microtubules in a living cell; namely, the strings evolve via a competition between (local) growth/shrinking as well as (global) hydrolysis processes. We give a complete characterization of the phase diagram of the model, and derive several criteria of the transient and recurrent regimes for the underlying stochastic process.
\end{abstract}

\keywords{microtubules; transience; recurrence; phase transition; birth-and-death process} 

\subjclass[2000]{60K35, 82B41}

\maketitle
\section{Introduction}

Microtubules are important structural components of the cytoskeleton, which play a vital role in many processes in a living cell. Their unique ability of rapid growth and even more rapid shrinking (often called dynamical instability) is exploited by the nature to segregate chromosomes during cell division, and as such microtubules have been the subject of intensive study. At the same time, the high complexity of the involved processes turns experimental study of microtubules into a challenging task, with many key questions in the area remaining unanswered.

In a recent paper \cite{tAplKsRmMbC08} the authors suggested a simplified stochastic model of microtubule growth aimed at deriving the dynamical instability from the interplay of a small number of parameters (The actual behaviour of microtubules is much more complex, see e.g.\ review \cite{oVnCdJ01} and references there). Mathematically, the model represents microtubules as long polymers made from two types of monomers, $\oplus$ and $\ominus$ (guanosin triphosphate ({\sf GTP${}^+$}) and guanosin diphosphate ({\sf GDP${}^-$}) tubulin complexes), subject to several stochastic transformations occurring with fixed rates, namely: growth, i.e., attachment of $\oplus$~monomers to the active end (with the rate depending on the type of the extremal monomer), hydrolysis, i.e., irreversible transformation of a $\oplus$~monomer into a $\ominus$~monomer (independently of the state of all other monomers composing the microtubule), and depolymerisation/shrinking, i.e., spontaneous departure of the hydrolysed extreme monomer (for a formal definition, see Sect.~\ref{sec:model} below). The authors described analytically the limiting behaviour of the model in several particular cases but had to rely upon numerical simulations in ``the more biologically relevant case of intermediate parameter values''~\cite{tAplKsRmMbC08}.

Our aim here is to describe the phase diagram of this model, in particular, to give several equivalent characterisations of the phase boundary, the set in the parameter space separating the region of the unbounded growth of microtubules from that of the ``compact phase'', where the average microtubule length remains bounded. According to one of our main results (for a complete list and rigorous statements, see Sect.~\ref{sec:results} below), for every point in the parameter space (i.e., a collection of fixed rates) there is a well defined value of velocity of the position of microtubule's active end, and it is the zero-velocity set in the parameter space which separates the regions of unbounded growth (positive velocity) from that of ``compact phase'' (negative velocity).

\subsection{The model}\label{sec:model}
Following \cite{tAplKsRmMbC08}, we think of microtubules $\bs{m}$ as of long polymers consisting of $\oplus$~and $\ominus$~monomers, $\bs{m}=\dots m_{2}m_{1}m_{0}$, where $m_k\in\bl\{\oplus,\ominus\br\}$ for all $k\ge0$, with the ``extreme'' monomer $m_0$ located at the \emph{active} end of the microtubule. Initially all monomers are in the $\ominus$~state, and the time evolution of the microtubule (formally described below) guarantees that with probability one at every moment of time the microtubule contains at most a finite number of $\oplus$~monomers; it is thus convenient to describe the current state of a microtubule at time $t$ in terms of the {\sf position} $x_t$ of the extreme monomer $m_0$ and the {\sf head} (or the {\sf populated zone}, \cite{tAplKsRmMbC08}) $\bsw_t$ of the microtubule, defined as the \emph{shortest} word $m_{k}\dots m_{1}m_{0}$ such that all other monomers $m_{n}$, $n>k$, are in the $\ominus$~state. Since attachment of new monomers occurs at the active end of a microtubule, every non-empty head $\bsw_t$ spans between the active end of the microtubule and its left-most $\oplus$~monomer.

Let $\bl\{\oplus,\ominus\br\}$ be a two-symbol alphabet, and let $\hatcalW=\cup_{k\ge0}\bl\{\oplus,\ominus\br\}^k$ denote the collection of all possible finite words, including the empty one. We call a {\sf head} any word belonging to the set
\[
 \calW=\bl\{\varnothing\br\}\cup\bl\{\bsw=\oplus\bs{\hatw}\text{ with }\bs{\hatw}\in\hatcalW\br\}\,,
\]
so that every non-empty head $\bsw$ can be written as a finite word $w_k\dots w_0$ for some integer $k\ge0$ with its left-most monomer being in the $\oplus$~state, $w_k=\oplus$ (Here and below, if $\bs{\hatw'}=\hatw'_k\dots \hatw'_0$ and $\bs{\hatw''}=\hatw''_l\dots \hatw''_0$ are two finite words in $\hatcalW$, we write $\bs{\hatw'\hatw''}$ for the concatenated word $\hatw'_k\dots\hatw'_0\hatw''_l\dots\hatw''_0$ of $k+l+2$ symbols). It is convenient to decompose the set $\calW$ of all finite heads into a disjoint union
\begin{equation}
\label{eq:calWplusminus-def}
 \calW=\calW_+\cup\calW_-\,,\quad\text{ where }\quad\calW_+=\bl\{\bsw=w_{k}\dots w_0\in\calW:w_0=\oplus\br\}\,.
\end{equation}
In this decomposition the heads $\bsw\in\calW_+$ correspond to microtubules whose active monomer $m_0$ is in the $\oplus$~state, whereas the set of heads $\calW_-$ is associated with those microtubules for which $m_0=\ominus$; in particular, initially we have $m_k\equiv\ominus$ for all $k\ge0$, i.e., the head is empty and thus $\varnothing\in\calW_-$. Of course, every finite word $\bs{\hatw}\in\hatcalW$ corresponds to a unique head $\bsw=\bl<\bs{\hatw}\br>\in\calW$ obtained by removing all its $\ominus$~monomers to the left of the left-most $\oplus$~monomer in $\bs{\hatw}$; it is convenient to think of $\bl<\,\cdot\,\br>:\hatcalW\to\calW$ as a projection operator.

Similarly, for integer $m\ge0$, $\ell\ge0$  let $\proj\cdot{\ell} m:\calW\to\bl\{\oplus,\ominus\br\}^{m+1}$ be the projection operator such that
\begin{equation}
\label{eq:finite-projection}
\bsw=w_{k}\dots w_0\quad\mapsto\quad\bshatw\equiv\proj\bsw {\ell}m=\hatw_{m}\dots\hatw_0\,,
\end{equation}
where $\hatw_j=w_{\ell+j}$ for $j\in\{0,\dots,m\}$, and we assume that the word $\bshatw$ is extended with $\ominus$~monomers on the left if necessary, i.e., $\hatw_j=\ominus$ if $\ell+j>k$ for $j$ under consideration. If $\ell=0$, we shall use a simplified notation $\proj\bsw{}m$ for the word consisting of the $m+1$ right-most monomers in~$\bsw$, again extended on the left as necessary.

Our main object here is a continuous-time Markov process 
\begin{equation}
\label{eq:yt-def}
 \bsy_t\equiv(x_t,\bsw_t)\,, \quad t\ge0\,,
\end{equation}
taking values in $\calY\equiv\BbbZ\times\calW$, where $x_t\in\BbbZ$ is the position at time $t\ge0$ of the right-most monomer $w_0$ of the head $\bsw_t\in\calW$. We shall assume that initially the microtubule consists of an empty head located at the origin,
\begin{equation}
\label{eq:init-condn}
 \bsy_0=(x_0,\bsw_0)=(0,\varnothing)\,,
\end{equation}
and that the transitions of $\bsy_t$ are described in terms of fixed positive constants $\lambda^+$, $\lambda^-$ and $\mu$ as follows:
\begin{description}
\item[Attachment:] a $\oplus$~monomer attaches to the right end of the microtubule,
\[
 (x,\bsw)\mapsto\bl(x+1,\bsw\oplus\br)\,,
\]
at rate $\lambda^+$ if $\bsw\in\calW_+$ and rate $\lambda^-$ if $\bsw\in\calW_-$; of course, if $\bsw=\varnothing\in\calW_-$, the head $\bsw\oplus$ should be understood as $\bl<\bsw\oplus\br>\equiv\oplus$.

\item[Detachment:] if $\bsw\in\calW_-$, i.e., $\bsw=\varnothing$ resp.\ $\bsw=\bs{w'}\ominus$ with some $\bs{w'}\in\calW\setminus\bl\{\varnothing\br\}$, the microtubule shrinks at rate $\mu$, 
\[
\bl(x,\varnothing\br)\mapsto\bl(x-1,\varnothing\br) \quad\text{~resp.~}\quad\bl(x,\bs{w'}\ominus\br)\mapsto\bl(x-1,\bs{w'}\br)\,.
\]

\item[Conversion:] for a non-empty head $\bsw=w_k\dots w_0\in\calW$, let $J_{\bsw}$ denote the list of positions of all $\oplus$~monomers in $\bsw$, $J_{\bsw}=\{j\ge0:w_j=\oplus\}$; then every $w_j$ with $j\in J_{\bsw}$ hydrolyses, $w_j=\oplus\mapsto\ominus$, at rate~$1$, independently of all other $w_j$, $j\in J_{\bsw}$. In other words, if $\bs{\hatw}$ is any word obtained from $\bsw$ by converting one of its $\oplus$~monomers into the $\ominus$~state, then at rate~$1$,
\[
 (x,\bsw)\mapsto\bl(x,\left<\bs{\hatw}\right>\br)\,,
\]
where transformations into different resulting words $\bs{\hatw}$ are independent. Notice that if the left-most $\oplus$~monomer $w_k$ in $\bsw$ hydrolyses, the resulting word $\bs{\hatw}$ starts with $\ominus$, so that in this case the new head $\left<\bs{\hatw}\right>$ is shorter than $\bsw$ and might even become empty.
\end{description}

In our analysis of the main microtubule process $(\bsy_t)_{t\ge0}$ we shall rely upon two auxiliary processes approximating $\bsy_t$.

Let $0=\tiltau_0<\tiltau_1<\dots$ be the moments of consecutive returns of the Markov process $(\bsy_t)_{t\ge0}$ to states with empty head,
\begin{equation}
\label{eq:tily-def}
\bs\tily_\ell\equiv\bsy_{\tiltau_\ell}=\bl(\tilx_\ell,\varnothing\br)\,,\qquad\ell\ge0\,.
\end{equation}
Clearly, the discrete time Markov chain $\bl(\bs\tily_\ell\br)_{\ell\ge0}$ can be identified with the process $(\bs\tilx_\ell)_{\ell\ge0}$, where $\bs\tilx_0=0$. Put $\theta_\ell=\tiltau_\ell-\tiltau_{\ell-1}$, $\ell>0$. As the differences
\[
\bl(\tilx_\ell-\tilx_{\ell-1},\theta_\ell\br)\equiv\bl(x_{\tiltau_\ell}-x_{\tiltau_{\ell-1}},\tiltau_\ell-\tiltau_{\ell-1}\br)
\]
are mutually independent and have the same distribution, the process $(\tilx_\ell)_{\ell\ge0}$ is a discrete time random walk on $\BbbZ$ with i.i.d.\ increments.

Our second auxiliary process is a ``finite-state version'' of the process $(\bsy_t)_{t\ge0}$. For a fixed integer $m\ge0$, let $[\,\cdot\,]^m:\calW\to\bl\{\oplus,\ominus\br\}^{m+1}$ be the projection operator defined above. We then put
\begin{equation}
\label{eq:haty-def}
 \bs\haty_t\equiv\bl(\hatx^m_t,\bs\hatw^m_t\br)\equiv\bs\haty^{\,m}_t\DF\bl[\bsy_t\br]^m=\bl(x_t,\bl[\bsw_t\br]^m\br)\,,
\end{equation}
and equip the process $\bs\haty_t$ with jumps (and rates) inherited from the process $\bsy_t$; then the conversion move for $\bs\haty_t$ is the same as for $\bsy_t$, whereas the attachment move should be understood as
\[
 (x,\bs{\hatw})\mapsto\bl(x+1,\left[\bs{\hatw}\oplus\right]^m\br)\,,
\]
and the detachment move becomes
\[
\bl(x,\varnothing\br)\mapsto\bl(x-1,\varnothing\br) \quad\text{~or~}\quad\bl(x,\bs{\hatw}\ominus\br)\mapsto\bl(x-1,\left[\ominus\bs{\hatw}\right]^m\br)\,,
\]
if $\bs{\hatw}$ contains at least one $\oplus$~monomer (in the ``finite-state'' situation here and below, $\varnothing$ denotes the word of length $m+1$ made of $\ominus$ monomers only). As a result, for every fixed $m\ge0$ the process $(\bs\haty_t)_{t\ge0}\equiv(\bs\haty^{\,m}_t)_{t\ge0}$ is a continuous-time Markov chain on a ``finite'' strip $\BbbZ\times\bl\{\oplus,\ominus\br\}^{m+1}$. The transience and recurrence properties of such chains are similar to those of discrete-time chains on strips, see e.g.  \cite[Chap.~3]{gFvaMmvM95}.

\subsection{Results}\label{sec:results}

We now are ready to state our main results.

\begin{thm}\label{thm:main-renewal}\sl
 The random vectors
\begin{equation}
\label{eq:renewal-increments}
 \bl(\Delta_l\tilx,\Delta_l\tiltau\br)\equiv\bl(\tilx_l-\tilx_{l-1}, \tiltau_l-\tiltau_{l-1}\br)\,,\qquad l\ge1\,,
\end{equation}
share a common distribution with finite exponential moments in a neighbourhood of the origin. Consequently, the discrete time random walk $(\tilx_l)_{l\ge0}$ in~$\BbbZ$, generated by the i.i.d.\ steps $\Delta_l\tilx$, satisfies all classical limiting results including the (Strong) Law of Large Numbers, the Central Limit Theorem and the Large Deviation Principle.
\end{thm}

Since the increments of the sequence $(\tiltau_l)_{l\ge0}$ have exponential moments, the embedded random walk $(\tilx_l)_{l\ge0}$ captures the long-time behaviour of the main process $(\bsy_t)_{t\ge0}$. In what follows we shall often say that the process $(\bsy_t)_{t\ge0}$ is transient towards~$+\infty$ (resp. $-\infty$) if the random walk $(\tilx_l)_{l\ge0}$ has the corresponding property.

\begin{cor}\label{cor:velocity}\sl
 The velocity $v$ of the process $(x_t)_{t\ge0}$, defined as the almost sure limit
\[
 v\DF\lim_{t\to\infty}\frac{x_t}t\,,
\]
satisfies $v=\sfE\tilx_1/\sfE\tiltau_1$. In particular, $\sfE\tilx_1>0$ corresponds to transience of $x_t$ towards $+\infty$ and $\sfE\tilx_1<0$ corresponds to transience of $x_t$ towards $-\infty$.
\end{cor}

\begin{rem}\sl
In a more realistic from the biological point of view case of $x_t$ restricted to the half-line $\BbbZ^+\equiv\bl\{0,1,2,\dots\br\}$, the condition $\sfE\tilx_1>0$ corresponds to unbounded growth (with speed $v>0$), whereas the condition $\sfE\tilx_1<0$ corresponds to the ``compact phase'' of positive recurrence.
\end{rem}

\begin{rem}\sl
Existence of exponential moments for the vectors \eqref{eq:renewal-increments} allows for a fast numerical estimation of $\sfE{\tilx_1}$, and thus provides a constructive way of describing the phase boundary for the Markov process $(\bsy_t)_{t\ge0}$.
\end{rem}

We now give another characterisation of the transient regime towards $+\infty$.

\begin{thm}\label{thm:finite-head}\sl
If $\lambda^-\ge\mu+\lambda^+$, the Markov process $(\bsy_t)_{t\ge0}$ is transient towards~$+\infty$.
Alternatively, if $\lambda^-<\mu+\lambda^+$, the process $(\bsy_t)_{t\ge0}$ is transient towards~$+\infty$ if and only if for some $m\ge0$ the $m$-projected process $(\bs\haty^{\,m}_t)_{t\ge0}$ is transient towards~$+\infty$.
\end{thm}

\begin{rem}\sl
The transience and recurrence properties of the ``finite-strip'' process $(\bs\haty^{\,m}_t)_{t\ge0}$ can be easily described through the solution $\bshatpi^m$ to a finite system of linear equations, see \eqref{eq:finite-m-chain} below. This, together with the fact that the Markov process $(\bsy_t)_{t\ge0}$ is well approximated by $(\bs\haty^{\,m}_t)_{t\ge0}$ for $m$ large enough, for precise results see Section~\ref{sec:finite-chains}, provides a constructive way of describing the $+\infty$~transient regime of~$(\bsy_t)_{t\ge0}$.
\end{rem}

We now give an alternative description of the ``compact phase''. Let the Markov process $(\bsy_t)_{t\ge0}$ start from the empty-head initial condition \eqref{eq:init-condn}. Assuming that the first event results in arrival of a $\oplus$~monomer at position~$1$, we define its arrival time via $\zeta_{01}\DF\min\bl\{t>0:x_t=1\br\}$, and, on the event when $\zeta_{01}$ is finite, the departure time $\zeta_{10}$ of this monomer, $\zeta_{10}\DF\min\bl\{t>\zeta_{01}:x_t=0\br\}$. Then the difference
\begin{equation}
\label{eq:T1-def}
 T_1\DF\zeta_{10}-\zeta_{01}>0
\end{equation}
describes the lifetime of the $\oplus$~monomer at position~$1$. Notice, that the lifetime of \emph{every} monomer attached to the microtubule after the initial time $t=0$ does not depend on the configuration of the microtubule at the moment of arrival and therefore has the same distribution as~$T_1$.

Of course, $T_1$ is almost surely finite, and we formally put $T_1=+\infty$ if either of the times $\zeta_{01}$, $\zeta_{10}$ is infinite. At the same time, we obviously have that $\sfE T_1>0$.

In what follows we shall see that it is the finiteness of $\sfE{T_1}$, which is central to describing the ``compact phase''; moreover, if $\sfE{T_1}<\infty$, then the Laplace transform $\varphi(s)\DF\sfE e^{-sT_1}$ of $T_1$ is finite for some $s<0$, equivalently, $T_1$ has finite exponential moments in a neighbourhood of the origin:

\begin{thm}\label{thm:lifetime-vs-transience}\sl
The Markov process $(\bsy_t)_{t\ge0}$ is transient towards $-\infty$ if and only if $\sfE{T_1}<\infty$, equivalently, $\varphi(s)$ is finite in a neighbourhood of the origin.
\end{thm}

The average lifetime $\sfE T_1$ can be computed from $\varphi(s)$ in the usual way, 
\[
 \sfE T_1\equiv-\lim_{s\downarrow0}\frac{d}{ds}\varphi(s)\,,
\]
and the latter has the following property.

\begin{lem}\label{lem:Laplace-of-lifetime}\sl
The Laplace transform $\varphi(s)\equiv\sfE e^{-sT_1}$ of the lifetime $T_1$ satisfies the following functional equation: for all $s\ge0$,
\begin{equation}
\label{eq:phi-eqn}
\varphi(s)
=\frac{\bl(1+\lambda^+\varphi(s)\br)\bl(\mu+\lambda^-\varphi(s)\br)}{(1+\lambda^++s)(\mu+\lambda^-+s)}
+\frac{\lambda^+\bl((\mu+s)\varphi(s)-\mu\br)}{(1+\lambda^++s)(\mu+\lambda^-+s)}\varphi(s+1)\,.
\end{equation}
\end{lem}

\begin{rem}\sl
In addition to being potentially useful for numerical evaluation of $\varphi(s)$, Lemma~\ref{lem:Laplace-of-lifetime} can be used to studying properties of various lifetimes, see Sect.~\ref{sec:lifetime-properties} below.
\end{rem}

The rest of the paper is devoted to the proofs of the above results. In Sect.~\ref{sec:renewal-structure} we shall use the intrinsic renewal structure of $(\bsy_t)_{t\ge0}$ and the regularity property of birth and death processes from Sect.~\ref{sec:regularity-BandD-processes} to derive Theorem~\ref{thm:main-renewal} and Corollary~\ref{cor:velocity}. Then, in Sect.~\ref{sec:finite-chains} we shall investigate stochastic monotonicity properties of the processes $(\bs\haty^m_t)_{t\ge0}$ and use them to verify the finite-strip approximation result, Theorem~\ref{thm:finite-head}. Finally, in Sect.~\ref{sec:lifetime-properties}, we prove Lemma~\ref{lem:Laplace-of-lifetime}, Theorem~\ref{thm:lifetime-vs-transience}, and establish a similar characterisation of the transience towards~$+\infty$, based upon the Large Deviation estimate from Sect.~\ref{sec:long-jump-density}.

%%%%%%%%%%%%%%%%%%%%%%%%%%%%%%%%%%%%%%%%%%%%%%%%%%%%%%%%%%%%%%%%%%%%%
\section{Renewal structure}\label{sec:renewal-structure}

In this section we shall exploit the intrinsic renewal property of the Markov process $(\bsy_t)_{t\ge0}$ related to the consecutive moments $\tiltau_l$, $l\ge0$, when its head becomes empty, $\bsw_{\tiltau_l}=\varnothing$. By the strong Markov property, for every fixed $l>0$, the process $\bsy_{\tiltau_l+t}\equiv\bl(x_{\tiltau_l+t},\bsw_{\tiltau_l+t}\br)$, $t\ge0$, has the same law as $(\bsy_t)_{t\ge0}$ if started from the initial state $\bl(x_{\tiltau_l},\varnothing\br)$; in addition, this law does not depend on the behaviour of $(\bsy_t)_{t\in[0,\tiltau_l)}$. It is thus sufficient to study $(\bsy_t)_{t\ge0}$ over a single cycle interval $[0,\tiltau_1)$.

%%%%%%%%%%%%%%%%%%%%%%%%%%%%%%%%%%%%%%%%%%%%%%%%%%%%%%%%%%%%%%%%%%%%%
\subsection{A single cycle behaviour}\label{sec:single-cycle}

Fix arbitrary positive rates $\lambda^+$, $\lambda^-$, and $\mu$, and denote $\lambda=\max(\lambda^+,\lambda^-)>0$. Our aim here is to relate the Markov process $(\bsy_t)_{t\ge0}$, whose dynamics is governed by the rates $\lambda^+$, $\lambda^-$, and $\mu$, to the continuous-time birth-and-death process $\bl(Y_t\br)_{t\ge0}$ with birth rate $\lambda$ and death rate~$1$ (per individual), see App.~\ref{sec:regularity-BandD-processes}. In fact, we shall couple them in such a way that the total number of $\oplus$~monomers in $\bsw_t\equiv\bsw(t)$,
\begin{equation}
\label{eq:head-norm-def}
 \|\bsw_t\|\DF\sum_{j\ge0}\Ind{w_{j}(t)=\oplus}
\end{equation}
is bounded above by the total number of individuals $Y_t$ in the birth-and-death process; in particular, the moment $\tiltau_1$ of the first disappearance of the head of the process~$\bsy_t$ shall occur no later than the first return of~$Y_t$ to the origin.

\subsubsection*{Formal construction}

We couple the Markov process $(\bsy_t)_{t\ge0}$ with rates $(\lambda^+,\lambda^-,\mu,1)$ described above and the birth-and-death process $\bl(Y_t\br)_{t\ge0}$ with death rate~$1$ and  birth rate $\lambda=\max(\lambda^+,\lambda^-)$, see App.~\ref{sec:regularity-BandD-processes}, in such a way that the inequality $0\le\|\bsw_t\|\le Y_t$ is preserved for all times $t\ge0$.

To start, we shall assume that $\lambda^-\ge\lambda^+$ so that $\lambda=\lambda^-$ (the necessary changes needed in the case $\lambda^-\le\lambda^+$ shall be commented on below), and then denote
\begin{equation}
\label{eq:lambda-rates-coupling}
\lambda_0\DF\min(\lambda^-,\lambda^+)>0\,,\qquad\delta\lambda\DF\bl|\lambda^+-\lambda^-\br|\ge0\,.
\end{equation}
Two cases need to be considered separately.

\bigskip\noindent{\sf Case~I}: Let $\bsy_t=(x_t,\bsw_t)$ with $\bsw_t\in\calW_-$, and $Y_t=n+\|\bsw_t\|$ with $n\ge0$. We define four independent exponentially distributed random variables $\zeta_1\sim\Exp(\lambda)$, $\zeta_2\sim\Exp(\mu)$, $\zeta_3\sim\Exp\bl(\|\bsw_t\|\bl)$, $\zeta_4\sim\Exp(n)$ and put $\zeta\equiv\min\bl(\zeta_1,\zeta_2,\zeta_3,\zeta_4\br)$. Then the first transition occurs at time $t+\zeta$ and is given by
\begin{itemize}
 \item if $\zeta=\zeta_1$, then $Y_{t+\zeta}=Y_t+1$ and a $\oplus$~monomer attaches to the microtubule, i.e.,  $\bsy_{t+\zeta}=(x_t+1,\bs{w'})$ with $\bs{w'}=\bsw_t\oplus$;
 \item if $\zeta=\zeta_2$, then $Y_{t+\zeta}=Y_t$ and the extremal $\ominus$~monomer $w_0(t)$ leaves the microtubule (i.e., for $\bsw_t=\varnothing$ we have $\bsy_{t+\zeta}=(x_t-1,\varnothing)$, and in the case $\bsw_t\in\calW_-\setminus\{\varnothing\}$ we have  $\bsy_{t+\zeta}=(x_t-1,\bs{w'})$ with $w'_{j}=w_{j+1}(t)$ for all $j\ge0$);
 \item if $\zeta=\zeta_3$, then $Y_{t+\zeta}=Y_t-1$ and one $\oplus$~monomer in $\bsw_t$ hydrolyses (uniformly at random);
 \item if $\zeta=\zeta_4$, then $Y_{t+\zeta}=Y_t-1$ and $\bsy_{t+\zeta}=\bsy_t$.
\end{itemize}

By using the well-known properties of exponential random variables, it is immediate to verify that jumps of both processes $(\bsy_t)_{t\ge0}$ and $(Y_t)_{t\ge0}$ have correct distributions. We also notice that the last transition above can only happen if $n>0$; as a result, the key inequality $\|\bsw_{t+\zeta}\|\le Y_{t+\zeta}$ is preserved after the jump.

\bigskip\noindent{\sf Case~II}: Let $\bsy_t=(x_t,\bsw_t)$ with $\bsw_t\in\calW_+$ and let $Y_t=n+\|\bsw_t\|$ with $n\ge0$. We now consider four independent exponential random variables $\zeta_1\sim\Exp(\lambda_0)$, $\zeta_2\sim\Exp(\delta\lambda)$, $\zeta_3\sim\Exp\bl(\|\bsw_t\|\bl)$, $\zeta_4\sim\Exp(n)$ (recall~\eqref{eq:lambda-rates-coupling}), and put $\zeta\equiv\min\bl(\zeta_1,\zeta_2,\zeta_3,\zeta_4\br)$. Then the first transition occurs at time $t+\zeta$ and is given by
\begin{itemize}
 \item if $\zeta=\zeta_1$, then $Y_{t+\zeta}=Y_t+1$ and a $\oplus$~monomer attaches to the microtubule, i.e.,  $\bsy_{t+\zeta}=(x_t+1,\bs{w'})$ with $\bs{w'}=\bsw_t\oplus$;
 \item if $\zeta=\zeta_2$, then $Y_{t+\zeta}=Y_t+1$ and $\bsy_{t+\zeta}=\bsy_t$;
 \item if $\zeta=\zeta_3$, then $Y_{t+\zeta}=Y_t-1$ and one $\oplus$~monomer in $\bsw_t$ hydrolyses (uniformly at random);
 \item if $\zeta=\zeta_4$, then $Y_{t+\zeta}=Y_t-1$ and $\bsy_{t+\zeta}=\bsy_t$.
\end{itemize}
Again, it is straightforward to check that the transitions above provide a correct coupling.

By using an appropriate case at every step, we construct a correct coupling of two processes $Y_t$ and $\bsy_t$ for all $t\ge0$ in the region $\lambda^-\ge\lambda^+$. The construction for $\lambda^-<\lambda^+$ is similar with the only difference that the simultaneous moves with rate $\lambda=\lambda^+$, i.e.,
$\left|\,\cdots\oplus\right>\,\mapsto\,\left|\,\cdots\oplus\oplus\right>$ and $Y_{t+\zeta}=Y_t+1$, occur when $w_0(t)=\oplus$ ({\sf Case~II}), whereas a pair of moves with rates $\lambda_0$ and $\delta\lambda$ (recall lines $\zeta=\zeta_1$ and $\zeta=\zeta_2$ in {\sf Case~II} above) occurs when $w_0(t)=\ominus$, i.e., in {\sf Case~I}.

Observe that in the coupling described above, every jump of the microtubule process $(\bsy_t)_{t\ge0}$ involving $\oplus$~monomers (attachment or hydrolysis) corresponds to an appropriate move (up or down) in the birth-and-death process~$\bl(Y_t\br)_{t\ge0}$. We shall use this coupling below to study the microtubule process~$(\bsy_t)_{t\ge0}$.

%%%%%%%%%%%%%%%%%%%%%%%%%%%%%%%%%%%%%%%%%%%%%%%%%%%%%%%%%%%%%%%%%%%%%
\subsection{Proof of Theorem~\ref{thm:main-renewal} and Corollary~\ref{cor:velocity}}\label{sec:proof-main-renewal}

To prove Theorem~\ref{thm:main-renewal}, fix positive jump rates $\lambda^+$, $\lambda^-$ and $\mu$ as above and consider the Markov process $(\bsy_t)_{t\ge0}$ starting from the initial condition~\eqref{eq:init-condn}, $\bsy_0=(0,\varnothing)$. Recall that $\tiltau_1>0$ is the first moment of time when the process $\bsy_t$ enters a state with empty head, $\bsy_{\tiltau_1}=(\tilx_1,\varnothing)$. Our aim is to show that the expectation
\[
 \Phi_0(z,s)\DF\sfE_{\bsy_0}\bl[\,z^{\tilx_1}e^{s\tiltau_1}\,\br]
\]
is finite for some $z>1$ and $s>0$. By using the Markov property at the end of the initial holding time $\eta_1\sim\Exp(\lambda^-+\mu)$, we deduce the relation
\begin{equation}
\label{eq:Phi0-fist-step-decomposition}
\Phi_0(z,s)=\frac{\lambda^-+\mu}{\lambda^-+\mu-s}\Bl[\frac{\mu}{\lambda^-+\mu}\,z^{-1}+\frac{\lambda^-}{\lambda^-+\mu}\,z\,\Phi_1(z,s)\Br]\,,
\end{equation}
where $\Phi_1(z,s)$ is defined as $\Phi_0(z,s)$ but with the initial condition $\bs{y'}=(1,\oplus)$, i.e., with the head $\bsw$ consisting of a single $\oplus$~monomer at position~$x=1$. It thus suffices to show that $\Phi_1(z,s)$ is finite for some $z>1$ and $s>0$.

To this end, we shall use the construction from Sect.~\ref{sec:single-cycle} to couple the microtubule process~$(\bsy_t)_{t\ge\eta_1}$ and the birth-and-death process $(Y_t)_{t\ge\eta_1}$ with birth rate $\lambda=\max(\lambda^+,\lambda^-)$, death rate~$1$ and the initial condition $Y_{\eta_1}=1$. Let $\bartau_1$ be the hitting time and let $\barkappa_1$ be the total number of jumps until the process $(\bsy_t)_{t\ge\eta_1}$ starting at $\bsy_{\eta_1}=(1,\oplus)$ hits an empty head state $(\tilx_1,\varnothing)$. Similarly, write $\tau_0$ for the hitting time and $\kappa_0$ for the total number of jumps until the birth-and-death process $Y_t$ hits the origin. Finiteness of $\Phi_1(z,s)$ shall follow from monotonicity of this coupling and the results of~App.~\ref{sec:regularity-BandD-processes}.

Let $\sharp$ be the total number of $\oplus$~monomers attached to the microtubule during the time interval~$[0,\tiltau_1)$. Since by the time $\tiltau_1$ all these $\oplus$~monomers have hydrolysed and some of the resulting $\ominus$~monomers might have detached from the microtubule, we obviously have $2\sharp\le\kappa_0+1$ and therefore
\[
\barkappa_1\le3\sharp\le\frac32(\kappa_0+1)\,,\qquad -1\le\tilx_1\le\sharp\le\frac{\kappa_0+1}2\,.
\]
It now follows from the inequality $\bartau_1\le\tau_0$ and Proposition~\ref{prop:regularity-B&D-process} that
\[
 \Phi_1(z,s)\equiv\sfE_1\Bl[\,z^{\tilx_1}\,e^{s\bartau_1}\,\Br]\le\sqrt{z}\,\sfE\Bl[\,z^{\kappa_0/2}\,e^{s\tau_0}\,\Br]<\infty\,,
\]
provided $\sqrt{z}\le\barz$ and $s\le\bars$, for some $\barz>1$ and $\bars>0$. This estimate together with the decomposition \eqref{eq:Phi0-fist-step-decomposition} implies the first claim of Theorem~\ref{thm:main-renewal}. The other results for the random walk $\tilx_n$ now follow in a standard way (\cite{wF50,aDoZ98}).
%\qed

\medskip
Notice that the argument above also proves the following result.

\begin{cor}\label{cor:barkappa0-exp-moments}\sl
Let $\bartau_1$ be the hitting time and let $\barkappa_1$ be the total number of jumps until the process $(\bsy_t)_{t\ge\eta_1}$ with initial state $\bsy_{\eta_1}=(1,\oplus)$ hits an empty head state $(\tilx_1,\varnothing)$. Then there exist $\barz>1$ and $\bars>0$ such that $\sfE_1\bl[\,z^{\barkappa_1}\,e^{s\bartau_1}\,\br]<\infty$ everywhere in the region $z\le\barz$ and $s\le\bars$.
\end{cor}

Of course, Corollary~\ref{cor:velocity}, the Strong Law of Large Numbers for the renewal scheme with increments \eqref{eq:renewal-increments}, follows immediately from Theorem~\ref{thm:main-renewal} (see, e.g., \cite[Sect.~5.2]{rD99}). In addition, the estimate $\max\limits_{t\in[0,\tiltau_1)}|x_t-x_0|\le\barkappa_1+1$ and the corollary above imply the corresponding concentration result, namely sharp exponential estimates for the probabilities of the events $\bl|\frac{\tilx_t}t-\frac{\sfE\tilx_1}{\sfE\tiltau_1}\br|>\varepsilon$ with small fixed $\varepsilon>0$ and large~$t$.

%%%%%%%%%%%%%%%%%%%%%%%%%%%%%%%%%%%%%%%%%%%%%%%%%%%%%%%%%%%%%%%%%%%%%
\section{Finite approximations}\label{sec:finite-chains}

By Theorem~\ref{thm:main-renewal}, the process $\bsw_t$ is an irreducible continuous time positive recurrent Markov chain in $\calW$. If $\bs\pi$ is its unique stationary distribution, denote
\[
 \bspip\equiv\bspi\bl(\calW_+\br)=\sum_{\bsw\in\calW_+}\bspi(\bsw)\,,\qquad
 \bspim\equiv\bspi\bl(\calW_-\br)=\sum_{\bsw\in\calW_-}\bspi(\bsw)\,,
\]
i.e., $\bspip$ (resp.~$\bspim$) is the probability that the right-most monomer in $\bsw$ is a $\oplus$~monomer (resp., a $\ominus$~monomer), recall~\eqref{eq:calWplusminus-def}.

Similarly, for every fixed $m\ge0$ the projected chain $\bshatw_t\equiv[\bsw_t]^m$ has a unique stationary distribution $\bshatpi^m$, for which we define
\[
 \bshatpi^m_+\equiv\sum_{\bsw\in\{\oplus,\ominus\}^m}\bshatpi^m(\bsw\oplus)\,,\qquad
 \bshatpi^m_-\equiv\sum_{\bsw\in\{\oplus,\ominus\}^m}\bshatpi^m(\bsw\ominus)\,.
\]
We then have the following result.

\begin{prop}\label{prop:velocity-ergodic-thm}\sl
 Denote ${v_+}\equiv\lambda^+>0$ and ${v_-}\equiv\lambda^--\mu$. Then, almost surely,
\begin{equation}
\label{eq:velocity-ergodic-thm}
 \lim_{t\to\infty}\frac1tx_t=\bspip\,{v_+}+\bspim\,{v_-}\,,\qquad
 \lim_{t\to\infty}\frac1t\hatx^m_t=\bshatpi^m_+\,{v_+}+\bshatpi^m_-\,{v_-}\,.
\end{equation}

As a result, if ${v_-}\ge {v_+}$, then the process $x_t$ is transient towards $+\infty$ and for every $m\ge0$ the process $\hatx^m_t$ is transient towards $+\infty$.
On the other hand, if ${v_-}<{v_+}$, then $x_t$ is transient towards $+\infty$ iff for all sufficiently large $m\ge0$ the process $\hatx^m_t$ is transient towards $+\infty$.
\end{prop}

Of course, if the velocity in the RHS of \eqref{eq:velocity-ergodic-thm} does not vanish, then $x_t$ is transient towards $+\infty$ or $-\infty$ depending on the sign of this velocity. We shall deduce Proposition~\ref{prop:velocity-ergodic-thm} below by first showing that $\bshatpi^m_+<\bshatpi^{m+1}_+<\bspip$ for all $m\ge0$ and then proving that in fact $\bspip=\lim_{m\to\infty}\bshatpi^m_+$.

\begin{rem}\sl
 Let ${\widehat\calQ}^m$ be the generator of the finite-state Markov chain ${\bshatw}_t$. By irreducibility, its stationary distribution $\bshatpi^m$ is the only probability distribution satisfying the finite-dimensional system of equations
\begin{equation}
\label{eq:finite-m-chain}
\bshatpi^m\,\widehat\calQ^m=0\,.
\end{equation}
The above proposition implies that if for some $m\ge0$ the solution to this system makes the RHS of \eqref{eq:velocity-ergodic-thm} positive, then both processes $(\bsy_t)_{t\ge0}$ and $(\bs\haty^m_t)_{t\ge0}$ are transient towards $+\infty$. This gives another numerical method of establishing transience towards $+\infty$ for the process $(\bsy_t)_{t\ge0}$.
\end{rem}

\begin{rem}\sl
If $m=0$, the stationary distribution $\bshatpi^m$ becomes
\[
\bshatpi^0=\bl(\bshatpi^0_+,\bshatpi^0_-\br)=\bl(\lambda^-/(1+\lambda^-),1/(1+\lambda^-)\br)
\]
and therefore the velocity $v_0\equiv\bshatpi^0_+\,{v_+}+\bshatpi^0_-\,{v_-}$ is non-negative iff
\begin{equation}
\label{eq:m=1.phase-boundary}
\mu\le\lambda^-\bl(1+\lambda^+\br)\,.
\end{equation}
It has been argued in \cite[Sect.~V.B]{tAplKsRmMbC08}, that the RHS of \eqref{eq:m=1.phase-boundary} provides an asymptotically correct approximation to the ``phase boundary'' $v=0$ in the limit of small $\lambda^+$ and $\lambda^-$. It is interesting to notice that according to Lemma~\ref{lem:strict-monotonicity-of-pimplus} below, every point in the phase space for which the equality in \eqref{eq:m=1.phase-boundary} holds, belongs to the region of positive velocity $v$ for the process $\bsy_t$; i.e., where $\bsy_t$ is transient towards~$+\infty$.
\end{rem}

We start by noticing that the explicit expressions \eqref{eq:velocity-ergodic-thm} for the limiting velocity of the process $x_t$ in Proposition~\ref{prop:velocity-ergodic-thm} follow directly from the Ergodic theorem for continuous time Markov chains. Indeed, for a fixed $m\ge0$, consider the Markov chain $\bs\haty_t\equiv\bs\haty^m_t$ with the initial conditions $\bs\haty_0=(0,\varnothing)$. Decomposing the difference $\hatx^m_t\equiv\hatx^m_t-\hatx^m_0$ into a sum of individual increments and rearranging, gives
\begin{equation}
\label{eq:xt-decomposition}
\hatx^m_t=\sum_{\bsw,\bs{w'}\in\calW^m}k_{\bsw,\bs{w'}}(t)\,\bl[x_{\bs{w'}}-x_{\bsw}\br]\,,
\end{equation}
where $k_{\bsw,\bs{w'}}(t)$ is the total number of transitions $\bsw\mapsto\bs{w'}$ for $\bsw_t$ during the time interval $[0,t]$. Of course,
\[
 x_{\bs{w'}}-x_{\bsw}=\begin{cases}
+1\,,& \quad\text{ if $\bs{w'}=[\bsw\oplus]^m$\,,}\\[.2ex]
-1\,,& \quad\text{ if $\bs{w'}=\ominus[\bsw]^m_1$\,,}\\[.2ex]
0\,,& \quad\text{ else\,,}
\end{cases}
\]
so it is sufficient to concentrate on the transitions which change the position of the microtubule active end. It is not difficult to deduce that the ratios $k_{\bsw,\bs{w'}}(t)/t$ converge to definite limits as $t\to\infty$. For example, fix $\bsw\in\calW^m_+$, $\bs{w'}=[\bsw\oplus]^m$ and write $k_{\bsw}(t)$ for the total number of visits to $\bsw$ for $\bs\hatw^m_t$ during the time interval $[0,t]$. If $T_{\bsw}$ denotes the first return time to state $\bsw$, then the Ergodic Theorem and the Strong Law of Large Numbers imply the almost sure convergence
\[
 \frac{k_{\bsw}(t)}t\to\frac1{\sfE_{\bsw}T_{\bsw}}\qquad\qquad
\frac{k_{\bsw,\bs{w'}}(t)}{k_{\bsw}(t)}\to\frac{\lambda^+}{\|\bsw\|+\lambda^+}\qquad
\text{ as $t\to\infty$\,.}
\]
Consequently, \cite[Chap.~3]{jrN97}
\[
 \frac{k_{\bsw,\bs{w'}}(t)}t\to\frac{\lambda^+}{(\|\bsw\|+\lambda^+)\sfE_{\bsw}T_{\bsw}}\equiv\bshatpi_{\bsw}\lambda^+\,,
\]
almost surely, as $t\to\infty$, where $\bshatpi=\bshatpi^m$ stands for the unique stationary distribution of the Markov chain~$\bs\hatw^m_t$. Repeating the same argument for all other pairs of states $\bsw$, $\bs{w'}$ in \eqref{eq:xt-decomposition} and re-summing, we deduce the second equality in \eqref{eq:velocity-ergodic-thm}.
A similar argument implies the velocity formula for the process $(\bsy_t)_{t\ge0}$.

Our main result here is the following observation:

\begin{lem}\label{lem:strict-monotonicity-of-pimplus}\sl
 Let positive rates $\lambda^+$, $\lambda^-$ and $\mu$ be fixed. Then for all integer $m\ge0$ we have $\bshatpi^m_+<\bshatpi^{m+1}_+$. Moreover, $\lim_{m\to\infty}\bshatpi^m_+=\bspip$.
\end{lem}

\begin{rem}\sl
If one interprets the projection operator~$[\,\cdot\,]^m$ as an enforced conversion $\oplus\mapsto\ominus$ outside a finite region, the statement of the lemma justifies the heuristics that {\sf ``a less strict enforcement policy increases the chances of seeing $\oplus$~monomers at the right end of microtubules''}.
\end{rem}

The remainder of this section is devoted to the proof of this lemma. We first establish a non-strict monotonicity of $\bshatpi^m_+$ in $m$ via a coupling argument in Sect.~\ref{sec:comparison-finite-chains}, and then deduce the strict monotonicity of $\bshatpi^m_+$ from a suitable probabilistic bound in Sect.~\ref{sec:strict-monotonicity-of-pimplus}. Finally, Sect.~\ref{sec:convergence-of-pimplus} is devoted to a proof of the convergence claim of Lemma~\ref{lem:strict-monotonicity-of-pimplus}.

%%%%%%%%%%%%%%%%%%%%%%%%%%%%%%%%%%%%%%%%%%%%%%%%%%%%%%%%%%%%%%%%%%%%%
\subsection{Comparison of finite chains}\label{sec:comparison-finite-chains}

Fix positive rates $\lambda^+$, $\lambda^-$ and $\mu$ and an integer $m\ge0$, and consider two Markov chains
\[
 \bsy'_t\equiv\bl(x'_t,\bsw'_t\br)\DF\bs{\haty}^m_t\qquad\text{ and }\qquad \bsy''_t\equiv\bl(x''_t,\bsw''_t\br)\DF\bs{\haty}^{m+1}_t
\]
with initial conditions $\bsy'_0=(0,\varnothing)$, $\bsy''_0=(0,\varnothing)$; in the finite size setting here and below, $\varnothing$ refers to a string of an appropriate length consisting of $\ominus$~monomers only.

We now construct a coupling of the processes $\bsy'_t$ and $\bsy''_t$ in such a way that for all $t\ge0$ the following monotonicity property holds:
\begin{equation}
\label{eq:words-ordering}
\bsw'_t\prec\bsw''_t\,,\quad\text{ i.e., }\quad {}^\forall k\quad w'_k(t)=\oplus\quad\Longrightarrow\quad w''_k(t)=\oplus\,.
\end{equation}
Of course, in view of the intrinsic renewal structure and the strong Markov property, it is sufficient to construct a coupling on a single cycle of the Markov chain $\bsy''_t$, i.e., on the time interval between two consecutive visits by~$\bsw''_t$ to the state $\varnothing$. Again, we shall proceed by defining a coupling of a single step at a time.

Recall that we use $\proj\cdot lm$ to denote the projection operator~$\proj\cdot lm:\calW\,\to\,\bl\{\oplus,\ominus\br\}^{m+1}$ from \eqref{eq:finite-projection}, and that  $\|\bsw_t\|$ is the total number of $\oplus$~monomers in $\bsw_t$, recall \eqref{eq:head-norm-def}. In our construction below we shall separately consider four different cases.

\bigskip\noindent{\sf Case~I}:
Let at time $t\ge0$ we have the following configuration:
\[
 \bsy'_t=(x'_t,\bsw')\,,\qquad \bsy''_t=(x''_t,\bsw'')\,,\qquad \bsw'\prec\bsw''\in\calW_-\,,
\]
so that both $\bsw'$ and $\bsw''$ end with a $\ominus$~monomer (in particular, we might have $\bsw'=\varnothing$ or $\bsw'=\bsw''=\varnothing$). Denote
\[
 J_0\DF\bl\{j\ge0:w'_{j}=w''_{j}=\oplus\br\}\,,\qquad
 J_1\DF\bl\{j\ge0:w'_{j}=\ominus\,,\ w''_{j}=\oplus\br\}\,,
\]
so that $\|\bsw'\|=|J_0|$, $\|\bsw''\|=|J_0\cup J_1|=|J_0|+|J_1|$. With $n_0=|J_0|$ and $n_1=|J_1|$ we consider four independent exponential random variables $\zeta_1\sim\Exp(\lambda^-)$, $\zeta_2\sim\Exp(n_0)$, $\zeta_3\sim\Exp(n_1)$, $\zeta_4\sim\Exp(\mu)$ and define $\zeta=\min(\zeta_1,\zeta_2,\zeta_3,\zeta_4)$. Then the next transition occurs at time $t+\zeta$ and is given by
\begin{itemize}
 \item if $\zeta=\zeta_1$, then a $\oplus$~monomer simultaneously attaches to both processes, ie, $\bsy'_{t+\zeta}=(x'_t+1,\proj{\bsw'\oplus}{}m)$ and $\bsy''_{t+\zeta}=(x''_t+1,\proj{\bsw''\oplus}{}{m+1})$;

 \item if $\zeta=\zeta_2$, then two $\oplus$~monomers $w'_{j}$ and $w''_{j}$, with $j\in J_0$ selected uniformly at random, hydrolyse simultaneously, i.e., $w'_{j}(t+\zeta)=\ominus$ and $w''_{j}(t+\zeta)=\ominus$, whereas all other monomers in $\bsw'$ and $\bsw''$ do not change; as a result, we have $x'_{t+\zeta}=x'_t$ and $x''_{t+\zeta}=x''_t$;

 \item if $\zeta=\zeta_3$, then the $\oplus$~monomer $w''_{j}$ with with $j\in J_1$ selected uniformly at random, hydrolyses, whereas all other monomers in $\bsw'$ and $\bsw''$ as well as $x$-components of both $\bsy$-processes do not change;

 \item if $\zeta=\zeta_4$, then the right-most $\ominus$~monomer detaches from both microtubules, in other words, $\bsy'_{t+\zeta}=(x'_t-1,\proj{\bsw'}1m)$ and $\bsy''_{t+\zeta}=(x''_t-1,\proj{\bsw''}1{m+1})$.
\end{itemize}

\bigskip\noindent{\sf Case~II}:
Let at time $t\ge0$ we have the following configuration:
\[
 \bsy'_t=(x'_t,\bsw')\,,\qquad \bsy''_t=(x''_t,\bsw'')\,,\qquad \bsw'\prec\bsw''\,,\quad \bsw'\in\calW_+\,,
\]
so that both $\bsw'$ and $\bsw''$ end with a $\oplus$-polymer. Defining index sets $J_0$ and $J_1$ and their cardinalities $n_0=|J_0|$ and $n_1=|J_1|$ as in {\sf Case~I}, we consider three independent exponential random variables $\zeta_1\sim\Exp(\lambda^+)$, $\zeta_2\sim\Exp(n_0)$, $\zeta_3\sim\Exp(n_1)$ and define $\zeta=\min(\zeta_1,\zeta_2,\zeta_3)$. Then the next transition occurs at time $t+\zeta$ and coincides with the corresponding $\zeta$-transition in {\sf Case~I}.

Our construction in the remaining case shall depend on which of the ``attachment'' parameters~$\lambda$ is bigger; we thus use the notations from \eqref{eq:lambda-rates-coupling},
\[
 \lambda_0\DF\min(\lambda^-,\lambda^+)>0\,,\qquad\delta\lambda\DF\bl|\lambda^+-\lambda^-\br|\ge0\,,
\]
and consider two sub-cases separately.

\bigskip\noindent{\sf Case~IIIa}:
Let $\lambda^-\ge\lambda^+$ and let at time $t\ge0$ we have the following configuration:
\[
 \bsy'_t=(x'_t,\bsw')\,,\qquad \bsy''_t=(x''_t,\bsw'')\,,\qquad \bsw'\prec\bsw''\,,\quad \bsw'\in\calW_-\,,\quad \bsw''\in\calW_+\,.
\]
Define index sets $J_0$ and $J_1$ and their cardinalities $n_0=|J_0|$ and $n_1=|J_1|$ as above, consider five independent exponential random variables, $\zeta_1\sim\Exp(\mu)$, $\zeta_2\sim\Exp(n_0)$, $\zeta_3\sim\Exp(n_1)$, $\zeta_4\sim\Exp(\lambda_0)$, $\zeta_5\sim\Exp(\delta\lambda)$ and put $\zeta=\min(\zeta_1,\zeta_2,\zeta_3,\zeta_4,\zeta_5)$. Then the next transition occurs at time $t+\zeta$ and is given by
\begin{itemize}
\item if $\zeta=\zeta_1$, then the right-most $\ominus$~monomer detaches from the head $\bsw'$, i.e., $\bsy'_{t+\zeta}=(x'_t-1,\proj{\bsw'}1m)$ and $\bsy''_{t+\zeta}=\bsy''_t$;

 \item if $\zeta=\zeta_2$, then two $\oplus$~monomers $w'_{j}$ and $w''_{j}$, with $j\in J_0$ selected uniformly at random, hydrolyse simultaneously, whereas all other monomers in $\bsw'$ and $\bsw''$ as well as $x$-components of both $\bsy$-processes do not change;

 \item if $\zeta=\zeta_3$, then a $\oplus$~monomer $w''_{j}$ with with $j\in J_1$ selected uniformly at random, hydrolyses, whereas all other monomers in $\bsw'$ and $\bsw''$ as well as $x$-components of both $\bsy$-processes do not change;

 \item if $\zeta=\zeta_4$, then a $\oplus$~monomer simultaneously attaches to both processes, ie, $\bsy'_{t+\zeta}=(x'_t+1,\proj{\bsw'\oplus}{}m)$ and $\bsy''_{t+\zeta}=(x''_t+1,\proj{\bsw''\oplus}{}{m+1})$;

 \item if $\zeta=\zeta_5$, then a $\oplus$~monomer attaches to the head $\bsw'_t$ only, in other words, $\bsy'_{t+\zeta}=(x'_t+1,\proj{\bsw'\oplus}{}m)$ and $\bsy''_{t+\zeta}=\bsy''_t$.
\end{itemize}

\bigskip\noindent{\sf Case~IIIb}:
If $\lambda^-<\lambda^+$ and the departing configuration is the same as in {\sf Case~IIIa}, we use the same construction as there with the only difference that for $\zeta=\zeta_5\sim\Exp(\delta\lambda)$, a $\oplus$~monomer attaches to $\bsy''_t$ only, i.e., $\bsy''_{t+\zeta}=(x''_t+1,\proj{\bsw''\oplus}{}{m+1})$ but $\bsy'_{t+\zeta}=\bsy'_t$.

\begin{lem}\label{lem:coupling-stable}\sl
 Let positive rates $\lambda^+$, $\lambda^-$, $\mu$ and an integer $m\ge0$ be fixed. Consider the truncated processes
\[
\bsy'_t=(x'_t,\bsw'_t)\DF\bs\haty^m_t\,,\qquad\bsy''_t=(x''_t,\bsw''_t)\DF\bs\haty^{m+1}_t
\]
starting from the ``empty'' initial conditions $\bsy'_0=(0,\varnothing)$, $\bsy''_0=(0,\varnothing)$.

If $\lambda^-\ge\lambda^+$, then for every fixed $t\ge0$ in the coupling above we either have $\bsw'_t=\bsw''_t$ or there exists a unique $j_0\ge0$ such that $\bsw''_{j_0}=\oplus$, $\bsw'_{j_0}=\ominus$ and
\[
\begin{gathered}
 \bsw'_{j}(t)=\bsw''_{j}(t)\qquad{}^\forall\, j<j_0\,,\qquad
\bsw'_{j}(t)=\bsw''_{j}(t)=\ominus\qquad{}^\forall\, j>j_0\,.
\end{gathered}
\]
\end{lem}

\begin{rem}\sl
In other words, Lemma~\ref{lem:coupling-stable} states that in the region $\lambda^-\ge\lambda^+$ for every $t\ge0$ the words $\ominus\bs\hatw^m_t$ and $\bs\hatw^{m+1}_t$ either coincide at all positions or have exactly one discrepancy at the position of the left-most $\oplus$~monomer in $\bs\hatw^{m+1}_t$.
\end{rem}

\begin{proof}
It is straightforward to verify that the claim of the lemma holds until the first visit to the state $(\bsw'_t,\bsw''_t)=(\varnothing,\oplus)$. Then in the cases $\zeta=\zeta_3$ or $\zeta=\zeta_5$ of {\sf Case~IIIa} we get $\bsw'_{t+\zeta}=\bsw''_{t+\zeta}$ (and coincides with $\varnothing$ or $\oplus$ respectively) and in the cases $\zeta=\zeta_1$ or $\zeta=\zeta_4$ the discrepancy remains of the same type (at a single place); clearly $n_0=0$ implies that $\zeta=\zeta_2$ does not happen with probability one. The result now follows from a straightforward induction.
\end{proof}

It remains to study the case $\lambda^+>\lambda^-$.

\begin{lem}\label{lem:coupling-unstable}\sl
 For fixed integer $m\ge0$ and positive rates $\lambda^+$, $\lambda^-$ and $\mu$, consider truncated processes $(\bsy'_t)_{t\ge0}$ and $(\bsy''_t)_{t\ge0}$ as defined in Lemma~\ref{lem:coupling-stable}.

If $\lambda^+\ge\lambda^-$, then for every fixed $t\ge0$ we have $\ominus\bsw'_t\prec\bsw''_t$, recall \eqref{eq:words-ordering}. Moreover, if these heads do not coincide ($\ominus\bsw'_t\neq\bsw''_t$), then there exists $j_1\ge j_0\ge0$ such that $\bsw'_{{j_0}}(t)=\bsw'_{{j_1}}(t)=\ominus$, $\bsw''_{{j_0}}(t)=\bsw''_{{j_1}}(t)=\oplus$, and
\[
\begin{gathered}
\bsw'_{j}(t)=\bsw''_{j}(t)\qquad{}^\forall\, j<j_0\,,\\[1ex]
\bsw'_{j}(t)=\ominus\qquad{}^\forall\, j>j_0\,,\qquad
\bsw''_{j}(t)=\ominus\qquad{}^\forall\, j>j_1\,.
\end{gathered}
\]
\end{lem}

\begin{rem}\sl
In other words, Lemma~\ref{lem:coupling-unstable} states that in the region $\lambda^+\ge\lambda^-$, for every fixed $t\ge0$ the words $\ominus\bsw'_t$ and $\bsw''_t$ either coincide, or there exists $j_0\ge0$ such that both strings $\ominus\bsw'_t$ and $\bsw''_t$ coincide to the right of  $j_0$ and $\bsw'_{j}(t)=\ominus$ for all $j\ge j_0$. Notice that in contrast to the case $\lambda^-\ge\lambda^+$, the discrepancy between $\ominus\bsw'_t$ and $\bsw''_t$ can now spread out over an interval containing several consecutive left-most $\oplus$~monomers in $\bsw''_t$.
\end{rem}

\begin{proof}
As in the case $\lambda^-\ge\lambda^+$, there is at most one discrepancy between $\ominus\bsw'_t$ and $\bsw''_t$ over the time interval until the first visit to the state $(\bsw'_t,\bsw''_t)=(\varnothing,\oplus)$; moreover, it can only happen at the position of the left-most $\oplus$~monomer in $\bsw''_t$.

We next observe that the joint dynamics described above guarantees that if at a moment $t\ge0$ the right-most monomers in $\bsw'_t$ and $\bsw''_t$ coincide, i.e., $w'_0(t)=w''_0(t)$, then both processes $\bsy'_t$ and $\bsy''_t$ shall run in parallel with all pairs of $\oplus$~monomers attached to $\bsw'$ and $\bsw''$ at times $t+s\ge t$ evolving identically till at least the first of the following events happens:
\begin{itemize}
\item the pair $(w'_0(t),w''_0(t))$ simultaneously leaves the process $(\bsy'_t,\bsy''_t)$;
\item the number of monomers to the right of the initial pair reaches $m$; at that moment $t+s$ we either have $\ominus\bsw'_{t+s}=\bsw''_{t+s}$ (if this pair has simultaneously hydrolysed by time $t+s$) or $\bsw'_{t+s}=[\bsw''_{t+s}]^m$  and $w''_{(m+1)}=\oplus$.
\end{itemize}

In other words, the discrepancy between $\bsw'_t$ and $\bsw''_t$  does not grow when $j_0>0$ and it can only grow when $j_0=0$, equivalently, when $\bsw''_t\in\calW_+$ and $\bsw'_t=\varnothing$; notice however, that this discrepancy can shorten or disappear at all due to simultaneous spontaneous hydrolysis of $\oplus$~monomers in both heads or due to the ``enforced'' hydrolysis on the left end of the heads. A straightforward induction now finishes the proof of the lemma.
\end{proof}

As a result, the joint dynamics described above gives the ordering \eqref{eq:words-ordering}, $\bsw'_t\prec\bsw''_t$ for all $t\ge0$, so that the Ergodic Theorem implies $\bshatpi^{m}_+\le\bshatpi^{m+1}_+$. Our next step is to establish the strict inequality in the last inequality.

%%%%%%%%%%%%%%%%%%%%%%%%%%%%%%%%%%%%%%%%%%%%%%%%%%%%%%%%%%%%%%%%%%%%%
\subsection{Strict monotonicity of $\bshatpi^m_+$}\label{sec:strict-monotonicity-of-pimplus}

To show the strict monotonicity of $\bshatpi^{m}_+$ as a function of $m$, we prove here that in addition to the ordering property \eqref{eq:words-ordering}, the coupling constructed in the previous section guarantees that the long-time density $\bshatpi^{m,m+1}_{\ominus,\oplus}$ of the moments when $\bsw'_t\in\calW_-$ and $\bsw''_t\in\calW_+$ is strictly positive. Then Lemma~\ref{lem:strict-monotonicity-of-pimplus} shall follow directly from the standard Ergodic theorem.

Our argument below shall be based upon the following three facts:
\begin{description}
 \item [{\sf Fact~I}]
Let $\tilT^{m+1}$ be the time between the consecutive returns by $\bs\hatw^{m+1}_t$ to the initial state $\varnothing$, and let $\sharp^{m+1}_t$ denote the number of returns by $(\bs\hatw^{m+1}_t)_{t\ge0}$ to~$\varnothing$ by time $t$. A straightforward generalization of the coupling in Sect.~\ref{sec:single-cycle} above shows that $\tilT^{m+1}$ is stochastically smaller than $\tiltau_1$, which by Theorem~\ref{thm:main-renewal} has exponential moments in a neighbourhood of the origin. Therefore, $\sfE\tilT^{m+1}\le\sfE\tiltau_1<\infty$ and the strong Markov property together with the LDP implies that the probability of the complement $\overline{\calA^1_t}$ to the event $\calA^1_t\DF\bl\{\omega:\sharp^{m+1}_t\ge\frac{t}{2\sfE\tiltau_1}\br\}$ decays exponentially fast as $t\to\infty$.

 \item [{\sf Fact~II}]
Consider now the joint dynamics of $(\bs\haty^m_t)_{t\ge0}$ and $(\bs\haty^{m+1}_t)_{t\ge0}$ starting from the ``empty head'' initial state $\bs\haty^m_0=(0,\varnothing)$, $\bs\haty^{m+1}_0=(0,\varnothing)$ as described above. Denote by $\hatpmmp$ the probability of the event
\[
 \Bl\{\omega:\text{ for some }s<\tilT^{m+1}\,,\ \bs\hatw^m_s=\varnothing\,,\ \bs\hatw^{m+1}_s=\oplus \Br\}\,.
\]
By a ``single-trajectory'' argument we easily deduce that $\hatpmmp>0$.

We next write $\tilT^{m+1}_k$ for the $k$th return of the pair $(\bs\hatw^m_t,\bs\hatw^{m+1}_t)_{t\ge0}$ to the state $(\varnothing,\varnothing)$, and use $\sharp^{\ominus\oplus}_k\equiv\sharp^{\ominus\oplus}(k)$ to denote the total number of returns to the state~$(\varnothing,\oplus)$ up to time $\tilT^{m+1}_k$ by the pair of processes $(\bs\hatw^m_t,\bs\hatw^{m+1}_t)_{t\ge0}$. Using the strong Markov property together with the Large Deviation Principle for binomial random variables we deduce that the probability of the complement $\overline{\calA^2_k}$ to the event $\calA^2_k\DF\bl\{\omega:\sharp^{\ominus\oplus}_k\ge\,\frac\hatpmmp2 k\br\}$ decays exponentially fast as $k\to\infty$.

 \item [{\sf Fact~III}]
By the construction in Sect.~\ref{sec:comparison-finite-chains} above, the holding time $\hat\eta^m_{\ominus\oplus}$ of the process $(\bs\hatw^m_t,\bs\hatw^{m+1}_t)_{t\ge0}$ at the state $(\varnothing,\oplus)$ has exponential distribution with parameter $\nu=1+\mu+\max(\lambda^+,\lambda^-)>0$. Let $p_\nu>0$ be the probability of the event $\bl\{\omega:\hat\eta^m_{\ominus\oplus}>1/\nu\br\}$. Consider now $L$ separate visits by $(\bs\hatw^m_t,\bs\hatw^{m+1}_t)_{t\ge0}$ to the state $(\varnothing,\oplus)$ and denote by $\sharp^\nu_L$ the number of those of them whose holding times $\hat\eta^m_{\ominus\oplus}$ are larger than $1/\nu$. By the standard LDP we deduce that the probability of the complement $\overline{\calA^3_L}$ to the event $\calA^3_L\DF\bl\{\omega:\sharp^\nu_L\ge\frac{p_\nu}2 L\br\}$ decays exponentially fast as $L\to\infty$.
\end{description}

We now deduce the strict monotonicity of $\bshatpi^m_+$ in Lemma~\ref{lem:strict-monotonicity-of-pimplus}; to this end, consider the events
\[
 \begin{gathered}
  \calB^1_t\DF\Bl\{\omega:\sharp^{m+1}_t\ge\frac{t}{2\sfE\tiltau_1}\Br\}\,,
\qquad \calB^2_t\DF\Bl\{\omega:\sharp^{\ominus\oplus}_t\ge\,\frac\hatpmmp2\,\frac{t}{2\sfE\tiltau_1}\Br\}\,,
\\[.6ex]
\calB^3_t\DF\Bl\{\omega:\sharp^\nu_t\ge\frac{p_\nu}2\frac\hatpmmp2\,\frac{t}{2\sfE\tiltau_1}\Br\}\,.
 \end{gathered}
\]
It follows from the discussion above that the probabilities $\sfP\bl(\,\overline{\calB^1_t}\,\br)$, $\sfP\bl(\,\overline{\calB^2_t}\mid{\calB^1_t}\,\br)$, and $\sfP\bl(\,\overline{\calB^3_t}\mid{\calB^2_t}\,\br)$ decay exponentially fast as $t\to\infty$ (here and below we write $\overline\calA$ for the complement of the event $\calA$). On the event $\calB^1_t\cap\calB^2_t\cap\calB^3_t$ the total time spent at the state $(\varnothing,\oplus)$ by the trajectories $(\bs\hatw^m_s,\bs\hatw^{m+1}_s)_{0\le s\le t}$ is bounded below by $p_\nu\,\hatpmmp\,t/(16\,\nu\,\sfE\tiltau_1)$
for all $t$ large enough, $t\ge t_1$. On the other hand, by the elementary inequality
\[
  \sfP\bl(\,\overline{A\cap B\cap C}\,\br)\le\sfP\bl(\,\overline{A}\,\br)
+\sfP\bl(\,\overline{B}\mid A\br)+\sfP\bl(\,\overline{C}\mid B\,\br)
\]
and the estimates above, the probability of the complement to $\calB^1_t\cap\calB^2_t\cap\calB^3_t$ satisfies
\[
 \bar p\equiv\sfP\bl(\,\overline{\calB^1_t\cap\calB^2_t\cap\calB^3_t}\,\br)\le
\sfP\bl(\,\overline{\calB^1_t}\,\br)+\sfP(\,\overline{\calB^2_t}\mid{\calB^1_t}\,\br)+\sfP\bl(\,\overline{\calB^3_t}\mid{\calB^2_t}\,\br)\le\frac12
\]
provided $t$ is large enough, $t\ge t_2$.

We finally deduce that for all $t\ge\max(t_1,t_2)$ we have
\[
\begin{split}
\bshatpi^{m+1}_+-\bshatpi^m_+\equiv\bshatpi^{m,m+1}_{\ominus,\oplus}
&\DF\lim_{t\to\infty}\frac1t\int_0^t\Ind{\hatw^m_0(s)=\ominus}\Ind{\hatw^{m+1}_0(s)=\oplus}\,ds
\\[1ex]
&\,\ge\,\frac{p_\nu\,\hatpmmp}{16\,\nu\,\sfE\tiltau_1}(1-\bar p)
\ge\frac{p_\nu\,\hatpmmp}{32\,\nu\,\sfE\tiltau_1}>0\,.
\end{split}
\]
This finishes the proof of the strict monotonicity of $\bshatpi^m_+$ in Lemma~\ref{lem:strict-monotonicity-of-pimplus}.

%%%%%%%%%%%%%%%%%%%%%%%%%%%%%%%%%%%%%%%%%%%%%%%%%%%%%%%%%%%%%%%%%%%%%
\subsection{Convergence of $\bshatpi^m_+$}\label{sec:convergence-of-pimplus}

We first observe that an obvious modification of the construction in Sect.~\ref{sec:comparison-finite-chains} provides a coupling of the processes $\bsy'_t\equiv(x'_t,\bsw'_t)\DF\bs{\haty}^m_t$ and  $\bsy''_t\equiv\bsy_t$. Consequently, the Ergodic theorem implies that
\[
 \bspip-\bshatpi^m_+=\lim_{t\to\infty}\frac1t\,
\int_0^t\Ind{\bs\hatw^m_s\in\calW_-}\Ind{\bsw_s\in\calW_+}\,ds\ge0\,,
\]
so it remains to bound above the last integral. We shall do this by using an argument similar to that in Sect.~\ref{sec:strict-monotonicity-of-pimplus}.

Let an integer $m\ge0$ be fixed. As in \eqref{eq:tily-def}, we shall use $\tiltau_\ell$ to denote the moment of $\ell$th return to the state $\varnothing$ by the process $\bsw_t$ (by monotonicity of the coupling we then also have $\bs\hatw^m_{\tiltau_\ell}=\varnothing$). We shall say that the {\sf discrepancy event} occurs during $\ell$th cycle, if for some $t\in[\tiltau_{\ell-1},\tiltau_\ell)$ we have $\bl(\bsw_t,\bs\hatw^m_t\br)\in\calW_+\times\{\varnothing\}$, ie, at time $t$ the right-most monomer of $\bsw_t$ is a $\oplus$~monomer, whereas $\bs\hatw^m_t$ is empty. Of course, this is only possible if at some $s\in[\tiltau_{\ell-1},t)$ we had $\bsw_s=w_{m+1}\dots w_{1}w_0$ with $w_{m+1}=w_0=\oplus$ and during $[s,t)$ all monomers to the right of $w_{m+1}$ detached from $\bsw_s$ with $w_{m+1}$ still being in the $\oplus$~state.

By independence and memoryless property of the hydrolysis process for individual monomers, the probability of discrepancy event during any given cycle drops sharply as $m$ increases. Indeed, by the observation above, the discrepancy event cannot occur for cycles with less than $3(m+1)+2=3m+5$ transitions, whereas by Corollary~\ref{cor:barkappa0-exp-moments} the probability of the event $\bl\{\barkappa_1\ge 3m+5\br\}$ is exponentially small as a function of~$m$.

Let $t>0$ be fixed; write $\calD^m_t$ for the collection of all indices $\ell$ such that a discrepancy event occurs during $[\,\tiltau_{\ell-1},\tiltau_\ell\,)$. If $\ell_0\equiv\max\bl\{\ell:\tiltau_\ell\le t\br\}$, then 
\begin{equation}
\label{eq:Jmt-integral}
\calJ_m(t)\DF\,\int_0^t
\Ind{\bs\hatw^m_s\in\calW_-}\Ind{\bsw_s\in\calW_+}\,ds
\le\sum_{\ell\in\calD^m_t}\bl(\tiltau_\ell-\tiltau_{\ell-1}\br)+\bl(t-\tiltau_{\ell_0}\br)\,.
\end{equation}
Our aim here is to prove the following result.

\begin{lem}\label{lem:Jmt-lower-bound}\sl
For every $\varepsilon>0$ there exists $m\ge0$ large enough such that for some $A>0$ and $a>0$ one has $\sfP\bl(\calJ_m(t)\ge\varepsilon t\br)\le Ae^{-at}$ uniformly in $t\ge0$.
\end{lem}

In view of the trivial bound $\calJ_m(t)\le t$, the Borel-Cantelli lemma implies that for every fixed $\varepsilon>0$ we have, with probability one,
\[
 0\le\bspip-\bshatpi^m_+\le\limsup_{t\to\infty}\frac1t\calJ_m(t)\le2\varepsilon
\]
if only $m\ge m_\varepsilon$. It thus remains to verify the claim of the lemma.

Let an arbitrary $\varepsilon>0$ be fixed. We shall use the following three facts:

\begin{description}
 \item [{\sf Fact~I}]
Denote $\sharp_t\DF\min\bl\{\ell\ge0:\tiltau_\ell\ge t\br\}$. Since the differences $\tiltau_{\ell+1}-\tiltau_\ell$, $\ell\ge0$, are {\sf iid} random variables with the same distribution as $\tiltau_1$, Theorem~\ref{thm:main-renewal} implies that for every $\zeta>0$ there exist positive $A_1$ and $a_1$ such that
\[
 \sfP\Bl(\,\Bl|\sharp_t-\frac t{\sfE\tiltau_1}\Br|\ge\zeta t\,\Br)\le A_1e^{-a_1t}\qquad\text{ for all $t\ge0$\,.}
\]

 \item [{\sf Fact~II}]
For $\ell=1,\dots,\sharp_t$, let $\tilkappa^\ell_0$ be the total number of transitions of the jump chain during $\ell$th cycle, ie, for $t\in[\,\tiltau_{\ell-1},\tiltau_\ell\,)$. By the discussion above, if the discrepancy event occurs during $\ell$th cycle, we necessarily have $\tilkappa^\ell_0\ge 3m+5$. Denote
\[
\calK^m_t\DF\sum_{\ell=1}^{\sharp_t}\,\tilkappa^\ell_0\,\Ind{\tilkappa^\ell_0\ge 3m+5}\,.
\]
By Corollary~\ref{cor:SAK-LDP}, for every $\zeta_2>0$ small enough there exist $\zeta'_2\in(0,\zeta_2)$, $m\ge0$, $A_2>0$ and $a_2>0$ such that
\[
\sfP\bl(\,\calK^m_t\notin(\,\zeta'_2t,\zeta_2t\,)\,\br)\le A_2e^{-a_2t}\qquad\text{ for all $t\ge0$\,.}
\]

 \item [{\sf Fact~III}]
During every cycle, each holding time is exponentially distributed with parameter not smaller than $\nu=\min(\,1,\lambda^-+\mu\,)>0$. As a result, duration of every single cycle of $\kappa$ jumps is stochastically dominated by the sum of $\kappa$ {\sf iid} $\Exp(\nu)$ random variables.

Notice also that if $\eta_j\sim\Exp(\nu)$, $j=1,\dots, k$, are {\sf iid} random variables then, by the classical LDP, for every $\zeta_3>0$ there exist $A_3>0$ and $a_3>0$ such that for all $\kappa\ge0$,
\[
 \sfP\Bl(\,\sum_{j=1}^\kappa\eta_j\ge\Bl(\frac1\nu+\zeta_3\Br)\kappa\,\Br)\le A_3e^{-a_3\kappa}\,.
\]
\end{description}

Combining these observations, we deduce that $\calJ_m(t)$ from \eqref{eq:Jmt-integral} is stochastically smaller than $\sum_{j=1}^{\calK^m_t}\,\eta_j$, with $\eta_j\sim\Exp(\nu)$, $j\ge1$, being {\sf iid} random variables. Taking $\zeta_2=\varepsilon\nu/3$ and $\zeta_3=1/(2\nu)$, we deduce that for some $m\ge0$, $A>0$ and~$a>0$,
\[
 \sfP\Bl(\,\sum_{j=1}^{\calK^m_t}\,\eta_j\ge\frac\varepsilon2t\,\Br)\le Ae^{-a\zeta'_2t}\,,
\]
ie, the result of Lemma~\ref{lem:Jmt-lower-bound} holds. Consequently, $\lim\limits_{m\to\infty}\bshatpi^m_+=\bspip$, as claimed.

%%%%%%%%%%%%%%%%%%%%%%%%%%%%%%%%%%%%%%%%%%%%%%%%%%%%%%%%%%%%%%%%%%%%%
\section{Properties of the lifetimes}\label{sec:lifetime-properties}

%%%%%%%%%%%%%%%%%%%%%%%%%%%%%%%%%%%%%%%%%%%%%%%%%%%%%%%%%%%%%%%%%%%%%
\subsection{Proof of Lemma~\ref{lem:Laplace-of-lifetime}}

By the Markov property, the lifetime $\Top$ of the extreme $\oplus$~monomer at the origin can be rewritten as (recall \eqref{eq:T1-def})
\[
 \Top\equiv\min\Bl\{t>0:\bsy_t=(-1,\varnothing)\mid\bsy_0=(0,\oplus)\Br\}\,.
\]
Similarly, the lifetime $\Tom$ of the extreme $\ominus$~monomer satisfies
\[
 \Tom\equiv\min\Bl\{t>0:\bsy_t=(-1,\varnothing)\mid\bsy_0=(0,\varnothing)\Br\}\,.
\]
For $s\ge0$, consider the Laplace transforms of these times, $\phiop(s)\DF\sfE e^{-s\Top}$ and $\phiom(s)\DF\sfE e^{-s\Tom}$.

Suppose the process $\bsy_t$ starts from $\bsy_0=(0,\varnothing)$. After an exponential holding time $\eta_0\sim\Exp(\mu+\lambda^-)$, the extreme $\ominus$~monomer either departs from the system or a $\oplus$~monomer attaches to it, thus increasing the total lifetime by $\Top'+\Tom'$, where $\Top'$ and $\Tom'$ are independent and have the same distributions as $\Top$ and $\Tom$ respectively. As a result, the strong Markov property implies
\begin{equation}
\label{eq:phiom}
\phiom(s)\equiv\sfE\bl(e^{-s\eta_0}\br)\,\Bl[\frac\mu{\mu+\lambda^-}+\frac{\lambda^-}{\mu+\lambda^-}\,\phiom(s)\,\phiop(s)\Br]\,.
\end{equation}

Similarly, after a holding time $\eta_1\sim\Exp(1+\lambda^+)$, the initial configuration $\bsy_0=(0,\oplus)$ either becomes $(0,\varnothing)$ or $(1,\oplus\oplus)$. In the second case, after a time $\Top''\sim\Top$ the process $\bsy_t$ arrives either in $(0,\varnothing)$ or in $(0,\oplus)$, depending on whether the $\oplus$~monomer initially at the origin hydrolyses by time $\Top''$ or not. Consequently, if $\Topop$ denotes the lifetime of the head $\oplus\oplus$, we get
\[
\begin{split}
\sfE\Bl[e^{-s\Topop}\mid\Top''\Br]
&=e^{-s\Top''}\Bl[e^{-\Top''}\,\phiop(s)+\bl(1-e^{-\Top''}\br)\,\phiom(s)\Br]
\\[.5ex]
&=e^{-(s+1)\Top''}\bl(\phiop(s)-\phiom(s)\br)+e^{-s\Top''}\phiom(s)\,,
\end{split}
\]
and as a result,
\[
 \sfE e^{-s\Topop}=\phiop(s+1)\bl(\phiop(s)-\phiom(s)\br)+\phiop(s)\phiom(s)\,.
\]
Combining this with the first-step decomposition at time $\eta_1$,
\[
 \phiop(s)=\sfE e^{-s\eta_1}
\Bl[\frac1{1+\lambda^+}\phiom(s)+\frac{\lambda^+}{1+\lambda^+}\sfE e^{-s\Topop}\Br]
\]
we obtain
\begin{equation}
\label{eq:phiop}
\phiop(s)=\frac{\sfE e^{-s\eta_1}}{1+\lambda^+}\Bl[\phiom(s)
+\lambda^+\phiop(s+1)\bl(\phiop(s)-\phiom(s)\br)+\lambda^+\phiop(s)\phiom(s)\Br]\,.
\end{equation}

Finally, recalling that for $\eta\sim\Exp(\rho)$ we have $\sfE e^{-s\eta}=\rho/(\rho+s)$, we rewrite \eqref{eq:phiom} and \eqref{eq:phiop} as
\[
\left\{
\begin{aligned}
 (\mu+\lambda^-+s)\,\phiom(s)&=\mu+\lambda^-\phiop(s)\phiom(s)
\\[1ex]
(1+\lambda^++s)\,\phiop(s)&=\bl(1+\lambda^+\phiop(s)\br)\phiom(s)
+\lambda^+\bl(\phiop(s)-\phiom(s)\br)\phiop(s+1)\,.
\end{aligned}
\right.
\]
Getting rid of $\phiom(s)$, we deduce that $\phiop(s)$ satisfies \eqref{eq:phi-eqn}. This finishes the proof of Lemma~\ref{lem:Laplace-of-lifetime}. %\qed

\medskip
Differentiating \eqref{eq:phiom}, equivalently, the first equation in the last display, we immediately deduce the following fact.

\begin{cor}\label{cor:Etop-Etom}\sl
For all positive $\mu$, $\lambda^+$ and $\lambda^-$, we have  $1+\mu\,\sfE\Tom=\lambda^-\,\sfE\Top$; in particular, both $\sfE\Tom$ and $\sfE\Top$ are finite or infinite simultaneously.
\end{cor}

\begin{rem}\label{rem:Top-Tom-dominates}\rm
Our argument above implies that the lifetime $\Top$ stochastically dominates $\Tom$, i.e., $\sfP\bl(\Top>t\br)\ge\sfP\bl(\Tom>t\br)$ for all $t\ge0$.
\end{rem}

%%%%%%%%%%%%%%%%%%%%%%%%%%%%%%%%%%%%%%%%%%%%%%%%%%%%%%%%%%%%%%%%%%%%%
\subsection{Proof of Theorem~\ref{thm:lifetime-vs-transience}}

Our aim here is to verify the following fact.

\begin{prop}\label{prop:Top-exp-moments}\sl
 Let $\Top$ be the lifetime of the extreme $\oplus$~monomer, and let $v$ be the velocity of the process $x_t$ as described in Corollary~\ref{cor:velocity}. Then $v<0$ if and only if $\sfE\Top<\infty$. Moreover, if $v<0$, then $\Top$ has exponential moments in a neighbourhood of the origin.
\end{prop}

Of course, Theorem~\ref{thm:lifetime-vs-transience} follows directly from Corollary~\ref{cor:velocity} and Proposition~\ref{prop:Top-exp-moments}.

\begin{proof}
Let first $\sfE\Top<\infty$ and let the process $\bsy_t=(x_t,\bsw_t)$, $t\ge0$, start from $\bsy_0=(0,\varnothing)$. To deduce that $v<0$, consider a sequence of stopping times $S_0=0$, $S_k=\min\bl\{t>S_{k-1}:x_t=-k\br\}$, $k\ge1$. Of course, $\bl\{S_k\br\}$ is just a renewal sequence whose increments $S_k-S_{k-1}$ are independent and share the same distribution as~$\Tom$.

Consider the sub-walk $\tilde{\tilx}_k\DF\tilx_{S_k}$ of the random walk $\tilx_\ell$ corresponding to the consecutive moments when the head $\bsw_t$ becomes empty, recall~\eqref{eq:tily-def}. As in Sect.~\ref{sec:proof-main-renewal}, the strong law of large numbers implies that, with probability one, as $k\to\infty$, we have $\tilde\tilx_k/S_k\to-1/\sfE\Tom$. Combining this with Corollary~\ref{cor:velocity}, we deduce that $v=-\frac1{\sfE\Tom}<0$, and observe that by Corollary~\ref{cor:Etop-Etom} the condition $\sfE\Tom<\infty$ is equivalent to $\sfE\Top<\infty$.

We next assume that $v<0$ and deduce existence of exponential moments for $\Top$ in a neighbourhood of the origin. To this end, it is sufficient to verify the following property:

{\sl
 For every $v<0$ there exist positive constants $K$, $A$, and $a$ such that 
\begin{equation}
\label{eq:Top-tail-bound}
\sfP(\Top>Kn)\le Ae^{-an}\qquad\text{ for all $n\ge1$\,.}
\end{equation}
}
Indeed, for every $\alpha\in(0,a/K)$ the bound \eqref{eq:Top-tail-bound} implies that
\[
 \sfE e^{\alpha\Top}\le\alpha e^{\alpha K}\sum_{n=0}^\infty e^{\alpha Kn}\sfP(\Top>Kn)\le\frac{A\alpha e^{\alpha K}}{1-e^{\alpha K-a}}<\infty\,.
\]

It thus remains to derive property \eqref{eq:Top-tail-bound}. We begin by considering the random walk $\bs\tily_l=(\tilx_l,\bstilw_l)$ starting from $\bs\tily_0=(0,\varnothing)$ as in \eqref{eq:tily-def}. By Theorem~\ref{thm:main-renewal} and Corollary~\ref{cor:velocity}, for every $\zeta_1>0$ the large deviation probability
$\sfP\bl(\tilx_n>(v+\zeta_1)n\br)$ is exponentially small as $n\to\infty$. In particular, for $\zeta_1=|v|/2$ there exist positive constants $A_1$ and $a_1$ such that $\sfP\bl(\tilx_n>vn/2\bl)\le A_1e^{-a_1n}$ for all $n\ge1$.

Assume that the process $\bsy_t$ starts from $\bsy_0=(0,\oplus)$. Consider the collection $\taust_l$, $l\ge0$, of consecutive moments of time when $\bsy_t$ enters states with empty head, i.e., $\bsw_t=\varnothing$. Clearly, all variables $\taust_0>0$, $\taust_1-\taust_0$, $\taust_2-\taust_1$, \dots, are independent and have exponential moments in a neighbourhood of the origin; moreover, all but the first one share the common distribution with the stopping time $\tiltau_1$ from \eqref{eq:tily-def}. We denote $L_n\equiv\max\bl\{l\ge0:\taust_l\le Kn\br\}$ (where $L_n=-\infty$ if $\taust_0>Kn$), and introduce the event $\calB_n^1\equiv\bl\{L_n\ge\frac{2K}{\sfE\tiltau_1}n\br\}$. By the usual Large Deviation Principle estimate (similar to {\sf Fact~I} in Sect.~\ref{sec:convergence-of-pimplus}), the complement $\overline{\calB_n^1}$ of $\calB_n^1$ is exponentially small: for every $K>0$ there exist positive constants $A_2$ and $a_2$ such that
\[
 \sfP\bl(\,\overline{\calB_n^1}\,\br)\equiv\sfP\Bl(L_n<\frac{2K}{\sfE\tiltau_1}n\Br)\le A_2e^{-a_2n}
\]
for all $n\ge1$. To simplify the notations, we put $K=\sfE\tiltau_1/2$ and assume that the constants $A_2$ and $a_2$ are compatible with this choice. On the event $\calB_n^1$ we now have $L_n\ge n$, equivalently, $\taust_n\le Kn=n\sfE\tau_1/2$.

Let $x^*_0\equiv x_{\taust_0}$ be the position of the end of the microtubule at the first moment $\taust_0>0$ when the head $\bsw_t$ vanishes (recall that $\bsw_0=\oplus$). By Corollary~\ref{cor:barkappa0-exp-moments}, $x^*_0$ has exponential moments in a neighbourhood of the origin, so that for every $\zeta_3>0$ there exist positive $A_3$ and $a_3$ such that $\sfP\bl(x^*_0>\zeta_3 n\br)\le A_3e^{-a_3n}$ for all $n\ge1$.

We finally observe that on the event $\calB_n^1$ we have $\bl\{\Top>Kn\br\}\subseteq\bl\{x_{\taust_n}\ge0\br\}$, so that using the Markov property at the moment $\taust_0$, we obtain
\[
 \sfP\bl(x_{\taust_n}\ge0\br)=\sum_{k\ge0}\sfP\bl(x_{\taust_0}=k\br)\,\sfP\bl(x_{\taust_n}-x_{\taust_0}\ge-k\br)\,.
\]
Now, taking $\zeta=\min(\zeta_1,\zeta_3)$, we can bound the RHS above by
\[
\sum_{k=0}^{\zeta n}\sfP\bl(x_{\taust_0}=k\br) \,\sfP\bl(x_{\taust_n}-x_{\taust_0}\ge-\zeta n\br)+\sfP\bl(x_{\taust_0}>\zeta n\br)\le A_4e^{-a_4n}\,,
\]
where $A_4=A_1+A_3>0$ and $a_4=\min(a_1,a_3)>0$.

Putting all these estimates together, we get
\[
 \sfP\bl(\Top>Kn\br)\le\sfP\bl(\,\overline{\calB_n^1}\,\br)+\sfP\bl(\Top>Kn\mid\calB_n^1\br)\le\sfP\bl(\,\overline{\calB_n^1}\,\br)+\sfP\bl(x_{\taust_n}\ge0\br)
\le Ae^{-an}
\]
for all $n\ge1$, where $A=A_2+A_4>0$ and $a=\min(a_2,a_4)>0$. This finishes our proof of \eqref{eq:Top-tail-bound} and that of Proposition~\ref{prop:Top-exp-moments}.
\end{proof}

\appendix
%%%%%%%%%%%%%%%%%%%%%%%%%%%%%%%%%%%%%%%%%%%%%%%%%%%%%%%%%%%%%%%%%%%%%
\section{Regularity of birth and death processes}\label{sec:regularity-BandD-processes}

For fixed $\lambda>0$ and $\mu>0$, consider a continuous time birth and death process $Y_t$, $t\ge0$, whose birth rate is $\lambda$ and death rate per individual is $\mu$. In other words, $Y_t$ is a Markov process on $\BbbZ^+=\{0,1,2,\dots\}$, such that every jump from state $k\ge0$ to $k+1$ has rate $\lambda$, and jumps from $k>0$ to $k-1$ have rate $k\mu$. Let $\tau_0$ be the hitting time and let $\kappa_0$ be the total number of jumps until the Markov chain $Y_t$ hits the origin. For $z\ge0$ and $s\in\BbbR$ consider the function
\[
 \overline{\psi}_m(z,s)\DF\sfE_m\bl[z^{\kappa_0}e^{s\tau_0}\br]\,,
\]
where as usual $\sfE_m$ stands for the conditional expectation corresponding to the initial state $X_0=m>0$. Our aim here is to verify the following result:

\begin{prop}\label{prop:regularity-B&D-process}\sl
Let an integer $M$ satisfy $M\mu>\lambda$. Then there exists $\barz>1$ and $\bars>0$ such that $\max_{m=1,\dots,M}\overline{\psi}_m(z,s)$ is finite provided $z\le\barz$ and $s\le\bars$.
\end{prop}

Our proof of Proposition~\ref{prop:regularity-B&D-process} in Sect.~\ref{sec:BDP} shall be based upon two auxiliary results for finite state Markov chains (Sect.~\ref{sec:MC}) and random walks with negative drift (Sect.~\ref{sec:RW}).

%%%%%%%%%%%%%%%%%%%%%%%%%%%%%%%%%%%%%%%%%%%%%%%%%%%%%%%%%%%%%%%%%%%%%
\subsection{Finite state Markov chains}\label{sec:MC}

For a fixed integer $M>1$ put
\begin{equation}
\label{eq:SM-dSM-def}
\calS_M=\bl\{1,2,\dots,M\br\}\,,\qquad\partial\calS_M=\bl\{0,M+1\br\}\,,
\end{equation}
and let strictly positive numbers $p_m$, $q_m$, $\rho_m$ with $m\in\calS_M$ satisfy $p_m+q_m=1$ for all $m\in\calS_M$. Let $X_t$ be the continuous time random walk on $\barcalS_M=\calS_M\cup\partial\calS_M$ whose generator $\calQ=\bl(Q_{ij}\br)_{i,j=0}^{M+1}$ has the following entries:
\[
 Q_{ij}=\begin{cases} p_m\rho_m\,,& i=m,\quad j=m+1\,,\\q_m\rho_m\,,& i=m,\quad j=m-1\,,\\-\rho_m\,,& i=m,\quad j=m\,,\end{cases}
\qquad  {}^\forall m\in\calS_M\,,
\]
and $Q_{ij}=0$ for all other $i$, $j\in\barcalS_M$. In other words, $X_t$ is a continuous time Markov chain on $\barcalS_M$ with absorbing boundary $\partial\calS_M$, such that upon arrival at state $m\in\calS_M$ the chain waits a random time $\xi_m\sim\Exp(\rho_m)$ and afterwards jumps to $m+1$ or $m-1$ with probabilities $p_m$ and $q_m$ respectively. For $b\in\partial\calS_M$, let $\tau_b$ be the hitting time and let $\kappa_b$ be the total number of steps until the chain $X_t$ reaches state~$b$. For real $s$ and non-negative $z$, consider the functions
\begin{equation}
\label{eq:phim-0-M+1-def}
\begin{gathered}
 \varphi^0_m(z,s)\DF\sfE_m\bl[z^{\kappa_0}e^{s\tau_0}\Ind{\tau_0<\tau_{M+1}}\br]\,,
\\[1ex]
 \varphi^{M+1}_m(z,s)\DF\sfE_m\bl[z^{\kappa_{M+1}}e^{s\tau_{M+1}}\Ind{\tau_{M+1}<\tau_0}\br]\,,
\end{gathered}
\end{equation}
where as before $\sfE_m(\cdot)$ denotes the conditional expectation corresponding to the initial state $X_0=m\in\calS_M$. Clearly, the quantities
\[
 \varphi^0_m(1,0)\equiv\sfP_m\bl(\tau_0<\tau_{M+1}\br)\qquad\text{ and }\qquad\varphi^{M+1}_m(1,0)\equiv\sfP_m\bl(\tau_{M+1}<\tau_0\br)
\]
are both positive and add up to~$1$. Our aim here is to verify the following claim:

\begin{lem}\label{lem:general-finite-state-MC}\sl
 There exist $z_0>1$ and $s_0>0$ such that for $|z|\le z_0$ and $s\le s_0$,
\[
 \max_{m\in\calS_M}\Bl\{\varphi^0_m(z,s)\,,\varphi^{M+1}_m(z,s)\Br\}<1\,.
\]
\end{lem}

\begin{proof}
We start by observing that for $\xi\sim\Exp(\rho)$ and $s<\rho$ the exponential moment $\sfE e^{s\xi}$ of $\xi$ satisfies $\sfE e^{s\xi}=\rho/(\rho-s)$ with the RHS being a decreasing function of $\rho>0$. This implies that every holding time $\xi_m$ satisfies
\[
 \sfE e^{s\xi_m}\le\frac{\bar\rho}{\bar\rho-s}<\infty\qquad\text{ if only }\qquad s<\bar\rho\DF\min_m\rho_m>0\,.
\]

Next, for every fixed trajectory $X_t$ with $k$ jumps, where $k<\min(\kappa_0,\kappa_{M+1})$, its time duration is a sum of independent holding times at all visited states, so that the exponential moment of the total time duration of this trajectory is bounded above by $\bar\rho^k/(\bar\rho-s)^k$. Consequently, for every $m\in\calS_M$
\[
 \varphi^0_m(z,s)\le\sfE_m\Bl[\Bl(\frac{z\bar\rho}{\bar\rho-s}\Br)^{\kappa_0}\Br]\,,\qquad
 \varphi^{M+1}_m(z,s)\le\sfE_m\Bl[\Bl(\frac{z\bar\rho}{\bar\rho-s}\Br)^{\kappa_{M+1}}\Br]\,.
\]
We now observe that in view of the estimate (cf.~\cite[Lemma~10.11]{dW91})
\[
 \min_m\sfP_m\bl(\kappa_0\le M\br)\ge\bar p\DF\bl(\min_m(p_m,q_m)\br)^M>0
\]
the stopping time $\kappa_0$ has exponential tails, $\max_m\sfP_m\bl(\kappa_0>nM\br)\le(1-\bar p)^n$; since a similar estimate holds for $\kappa_{M+1}$, we deduce that $\max_m\Bl\{\sfE_m\bl[\bar z^{\,\kappa_0}\br],\sfE_m\bl[\bar z^{\,\kappa_{M+1}}\br]\Br\}$ is finite for some $\bar z>1$. Therefore, the estimate
\begin{equation}
\label{eq:varphi-0M-finite}
 \max_m\Bl\{\varphi^0_m(z,s), \varphi^{M+1}_m(z,s)\Br\}<\infty
\end{equation}
holds for all $s\le s'$ and $|z|\le z'$ with $s'>0$ and $z'\in(1,\bar z)$ satisfying the condition $\bar\rho z'/(\bar\rho-s')\le\bar z$, equivalently, $s'\le\bar\rho\bl(1-z'/\bar z\br)$. Since all functions in the LHS of \eqref{eq:varphi-0M-finite} are continuous for $z$ and $s$ in the region under consideration, and
\[
\max_m\Bl(\varphi^0_m(1,0),\varphi^{M+1}_m(1,0)\Br)\equiv\max_m\Bl(\sfP_m\bl(\tau_0<\tau_{M+1}\br),\sfP_m\bl(\tau_{M+1}<\tau_0\br)\Br)<1\,,
\]
the claim of the lemma follows.
\end{proof}

%%%%%%%%%%%%%%%%%%%%%%%%%%%%%%%%%%%%%%%%%%%%%%%%%%%%%%%%%%%%%%%%%%%%%
\subsection{Random walks with negative drift}\label{sec:RW}

For fixed $\lambda>0$ and $\nu>0$, let $X_t$ be the continuous-time homogeneous random walk on the half-line $\BbbZ^+=\bl\{0,1,2,\dots\br\}$ with absorption at the origin, whose jumps from state $k>0$ to $k+1$ have rate $\lambda$ and those from $k>0$ to $k-1$ have rate $\nu$. Let $\tau_0$ be the hitting time and let $\kappa_0$ be the total number of jumps until the Markov chain $X_t$ hits the origin. For $z\ge0$  and $s\in\BbbR$ consider the functions
\begin{equation}
\label{eq:psi-m-RW-negative-drift}
\psi_m(z,s)\DF\sfE_m\bl[z^{\kappa_0}e^{s\tau_0}\br]\,,\qquad m\in\BbbN\,.
\end{equation}
Our aim here is to verify the following claim:

\begin{lem}\label{lem:RW-negative-drift}\sl
 Let $\nu>\lambda$; if $s'>0$ and $z'>1$ are such that
\begin{equation}
\label{eq:negative-drift-RW-stability}
s'+2(z'-1)\sqrt{\lambda\nu}<\bl(\sqrt\nu-\sqrt\lambda\br)^2\,,
\end{equation}
then $\psi_1(z,s)<\infty$ for all $z\le z'$ and $s\le s'$.
\end{lem}

\begin{rem}\sl
Notice that if $\nu>\lambda$ (i.e., $X_t$ has negative drift), then for all $m\in\BbbN$ we have $\sfP_m\bl(\tau_0<\infty\br)=1$, and the lemma implies that $\psi_1(z,s)\searrow1$ as $z\searrow1$ and $s\searrow0$.
\end{rem}

Our proof below is a straightforward adaptation of the standard argument for the discrete-time walks (see, eg., \cite[Sect.~1.4]{jrN97}). We notice, however, that an alternative proof of Lemma~\ref{lem:RW-negative-drift} can be obtained by computing $\psi_m(z,s)$ explicitly. Namely, by conditioning on the jump chain, we deduce (similarly to the argument in Sect.~\ref{sec:RW})
\[
 \psi_m(z,s)=\sfE_m\Bl[\Bl(\frac{(\lambda+\nu)z}{\lambda+\nu-s}\Br)^{\kappa_0}\Br]
=\bigg(\sfE_1\Bl[\Bl(\frac{(\lambda+\nu)z}{\lambda+\nu-s}\Br)^{\kappa_0}\Br]\bigg)^m\,,
\]
so it remains to observe that the last expectation is finite iff $4\lambda\nu z^2\le(\lambda+\nu-s)^2$ (missing details behind the last two steps and the explicit expression for the generating function can be found in the classical monograph \cite[Sect.~14.4]{wF50}).

\begin{proof}
 Applying the Markov property at the moment of the first jump out of the initial state~$1$, we get
\[
 \psi_1(z,s)\equiv z\,\frac{\lambda+\nu}{\lambda+\nu-s}\Bl(\frac{\lambda}{\lambda+\nu}\,\psi_2(z,s)+\frac{\nu}{\lambda+\nu}\Br)\,.
\]
On the other hand, the strong Markov property implies that $\psi_m(z,s)\equiv\bl[\psi_1(z,s)\br]^m$, for all $m\in\BbbN$, so that $\psi_1(z,s)$ is given by the smallest positive solution $\psi$ to the quadratic equation $\lambda\psi^2+\nu=a\psi$ with $a=(\lambda+\nu-s)/z$. Such a solution exists and is finite iff $a^2\ge4\lambda\nu$, equivalently, if $\lambda+\nu-s\ge2z\sqrt{\lambda\nu}$; as $z>0$, the latter condition coincides with \eqref{eq:negative-drift-RW-stability}.
\end{proof}

%%%%%%%%%%%%%%%%%%%%%%%%%%%%%%%%%%%%%%%%%%%%%%%%%%%%%%%%%%%%%%%%%%%%%
\subsection{Proof of Proposition~\ref{prop:regularity-B&D-process}}\label{sec:BDP}

Our argument is based upon Lemmata~\ref{lem:general-finite-state-MC} and~\ref{lem:RW-negative-drift}, as well as on the following fact.

\begin{lem}\label{lem:BDP}\sl
Let $Y_t$, $t\ge0$, be the continuous time birth and death process with intensities $\lambda>0$ and $\mu>0$, as described above. Fix an integer $M>1$ such that $M\mu>\lambda$ and use $\hat\tau$ and $\hat\kappa$ to denote the hitting time and the total number of steps until the process $Y_t$ hits state $M$. Then there exist real numbers $\hat z>1$ and $\hat s>0$ such that the generating function
$
 \widehat\psi_M(z,s)\DF\sfE_{M+1}\bl[z^{\hat\kappa}e^{s\hat\tau}\br]
$
is finite
for all $s\le\hat s$ and $z\le\hat z$.
\end{lem}

\begin{proof}
Let $X_t$, $t\ge0$, be the continuous time simple random walk on $\BbbZ$ with upwards rate $\lambda$ and downwards rate $\nu\equiv M\mu>\lambda$. Coupling $X_t$ and $Y_t$ starting from the common state
$X_0=Y_0=M+1$
in a monotone way (e.g., by using the Harris construction), we get
$
 \widehat\psi_M(z,s)\le\psi_1(z,s)\,,
$
where $\psi_1(\cdot,\cdot)$ is determined as in \eqref{eq:psi-m-RW-negative-drift} for the random walk~$X_t$. The result now follows from Lemma~\ref{lem:RW-negative-drift}.
\end{proof}

We turn now to the proof of Proposition~\ref{prop:regularity-B&D-process}. Let an integer $M$ be as in Lemma~\ref{lem:BDP}, namely, let $M$ satisfy the condition $M\mu>\lambda>0$. We also fix an initial state $m\in\calS_M$, recall \eqref{eq:SM-dSM-def}. It is convenient to re-sum the parts of the trajectories of $Y_t$ connecting states $M+1$ and $M$, thus transforming the birth-and-death process $(Y_t)_{t\ge0}$ into a continuous-time finite-state Markov chain with the state space $\calS_M$ (recall \eqref{eq:SM-dSM-def}). We shall split all trajectories contributing to
\[
 \overline{\psi}_m(z,s)\equiv\sfE_m\bl[z^{\kappa_0}e^{s\tau_0}\br]
\]
into groups $\bfB_\ell$ with an integer $\ell\ge0$ specifying the number of transitions from state $M+1\in\partial\calS_M$ to state $M\in\calS_M$ before the trajectory hits the absorbing state $0\in\partial\calS_M$, see Fig.~\ref{fig:BD-expansion}.

\begin{figure}[htp]
\vskip3ex
\psfrag{A}{{\scriptsize$\bf B_0$}}
\psfrag{B}{{\scriptsize$\bf B_1$}}
\psfrag{C}{{\scriptsize$\bf{{B}_1'}$}}
\psfrag{T1}{$\frt_1$}
\psfrag{T2}{$\frt_2$}
\psfrag{U1}{$\fru_1$}
\psfrag{U2}{$\fru_2$}
\psfrag{0}{{\scriptsize{\sf0}}}
\psfrag{1}{{\scriptsize{\sf1}}}
\psfrag{M}{{\scriptsize{\sf M}}}
\psfrag{M+1}{{\scriptsize{\sf M'}}}
\centerline{\InputPictureHeight{40mm}{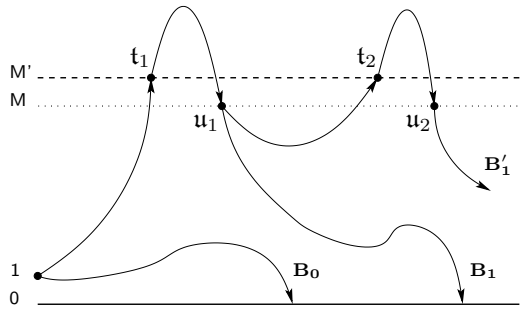}}
\vskip3ex
\caption{Three groups of trajectories: $\bf B_0$ -- trajectories hitting state~$0$ without visiting state $M'=M+1$; $\bf B_1$ -- trajectories visiting state $M'$ exactly once before hitting state~$0$; $\bf{{B}_1'}$ -- trajectories visiting state $M'$ more than once.}\label{fig:BD-expansion}
\end{figure}

Of course, on $\bfB_0$ we have 
$
 \sfE_m\bl[z^{\kappa_0}e^{s\tau_0}\one_{\bf B_0}\br]\equiv\varphi^0_m(z,s)
$
(recall \eqref{eq:phim-0-M+1-def}),
with the RHS being finite in the region specified by Lemma~\ref{lem:general-finite-state-MC}. Otherwise, the trajectory in question visits state $M+1$ at least once and thus both stopping times $\frt_1$ and $\fru_1$,
\[
 \frt_1\DF\min\bl\{t>0:Y_t=M+1\br\}\,,\qquad \fru_1\DF\min\bl\{u>\frt_1:Y_u=M\br\}
\]
are well defined. By the strong Markov property,
\[
 \sfE_1\bl[z^{\kappa_0}e^{s\tau_0}\one_{\bf B_1}\br]\equiv\varphi^{M+1}_1(z,s)\widehat\psi_M(z,s)\varphi^0_M(z,s)\,.
\]

Similarly, defining stopping times $\frt_\ell$, $\fru_\ell$, $\ell>1$, via
\[
 \frt_\ell\DF\min\bl\{t>\fru_{\ell-1}:Y_t=M+1\br\}\,,\qquad \fru_\ell\DF\min\bl\{u>\frt_\ell:Y_u=M\br\}\,,
\]
we deduce that for $\ell>1$
\[
 \sfE_m\bl[z^{\kappa_0}e^{s\tau_0}\one_{\bf B_\ell}\br]\equiv\varphi^{M+1}_m(z,s)\widehat\psi_M(z,s)\Bl[\varphi^{M+1}_M(z,s)\widehat\psi_M(z,s)\Br]^{\ell-1}\varphi^0_M(z,s)\,.
\]
As a result,
\[
\begin{split}
 \overline{\psi}_m(z,s)&\equiv\sum_{\ell\ge0}\sfE_m\bl[z^{\kappa_0}e^{s\tau_0}\one_{\bf B_\ell}\br]=\varphi^0_m(z,s)
\\[.5ex]
&\hphantom{\equiv\sum_{\ell\ge0}\sfE_1\bl[}
+\sum_{\ell>0}\varphi^{M+1}_m(z,s)\,\widehat\psi_M(z,s)\Bl[\varphi^{M+1}_M(z,s)\,\widehat\psi_M(z,s)\Br]^{\ell-1}\varphi^0_M(z,s)
\\[.7ex]
&=\varphi^0_m(z,s)+\frac{\varphi^{M+1}_m(z,s)\widehat\psi_M(z,s)\varphi^0_M(z,s)}{1-\varphi^{M+1}_M(z,s)\,\widehat\psi_M(z,s)}\,,
\end{split}
\]
provided the last expression is finite.

Finally, the product $\varphi^{M+1}_M(z,s)\,\widehat\psi_M(z,s)$ is continuous in the region
\[
 |z|\le\tilde z\DF\min(z_0,\hat z)\,,\qquad s\le\tilde s\DF\min(s_0,\hat s)\,,
\]
where by Lemmata~\ref{lem:general-finite-state-MC} and~\ref{lem:RW-negative-drift} we have $\tilde z>1$ and $\tilde s>0$. Since
\[
 \varphi^{M+1}_M(1,0)\,\widehat\psi_M(1,0)\le\sfP_M\bl(\tau_{M+1}<\tau_0\br)<1\,,
\]
the claim of the proposition follows by continuity.

%%%%%%%%%%%%%%%%%%%%%%%%%%%%%%%%%%%%%%%%%%%%%%%%%%%%%%%%%%%%%%%%%%%%%
\section{Long jumps density estimate for a class of random walks}\label{sec:long-jump-density}

Our aim here is to derive a simple estimate for random walks whose jumps have exponential moments in a neighbourhood of the origin. This observation is at the heart of our argument in Sect.~\ref{sec:convergence-of-pimplus}, but is also of independent interest. Of course, the statement and the proof below can  be generalized to continuous distributions.

Let $X_j$, $j\ge1$, be a sequence of {\sf iid} random variables with values in $\BbbN=\{1,2,\dots\}$, whose common distribution has finite exponential moments in a neighbourhood of the origin, i.e., $\sfE\bl[\bars^X\br]<\infty$ for some $\bars>1$. For a fixed $K\in\BbbN$, we think of $X_1$, \dots, $X_K$ as jumps of a random walk in $\BbbZ_+$, and for $A>0$ put $S^A_K\DF\sum_{j=1}^K\,X_j\Ind{X_j>A}$, i.e., $S^A_K$ is the total length of jumps $X_j$, $1\le j\le K$, longer than~$A$. We then have the following result.

\begin{lem}\label{lem:SAK-upper-bound}\sl
For every $\varepsilon>0$ there exist $A_0>0$, $K_0>0$ and $\alpha>0$ such that the inequality $\sfP\bl(S^A_K>\varepsilon K\br)\le e^{-\alpha K}$ holds for all $A\ge A_0$ and $K\ge K_0$.
\end{lem}

In view of the {\sf a priori} estimate $S^A_K\ge A\sum_{j=1}^K\Ind{X_j>A}$, the claim of the lemma and the standard LDP for Binomial random variables with parameters $K$ and $p_A=\sfP(X>A)$ imply the following observation:

\begin{cor}\label{cor:SAK-LDP}\sl
 For every $\varepsilon>0$ there exist $A_0>0$ and $\varepsilon_1\in(0,\varepsilon)$ such that the probability
\begin{equation}
\label{eq:SAK-LDP}
\sfP\bl(S^A_K\notin(\varepsilon_1K,\varepsilon K)\br)
\end{equation}
decays exponentially fast as $K\to\infty$.
\end{cor}

\begin{rem}\sl
 Of course, the very existence of two constants $0<\varepsilon_1<\varepsilon$ in \eqref{eq:SAK-LDP} is a straightforward consequence of the LDP. The main result in the lemma and the corollary above is that the velocity of the random walk $S^A_K$, $K\ge0$, vanishes asymptotically as $A\to\infty$.
\end{rem}

\begin{proof}[Proof of Lemma~\ref{lem:SAK-upper-bound}]
For a fixed $A>0$, put $\tilX_j\DF X_j\Ind{X_j>A}$. The integer-valued random variables $\tilX_j\ge0$ are {\sf iid} and satisfy, for $\bars>1$ as above, $\sfE\bl[\bars^{\tilX}\br]\le\sfE\bl[\bars^X\br]<\infty$. Moreover, by dominated convergence, $\sfE\tilX\equiv\sfE\bl[X\Ind{X>A}\br]\to0$ as $A\to\infty$, and we fix $A>0$ such that $\sfE\tilX<\varepsilon/2$. By the standard LDP, there exists $\tilde\alpha>0$ such that
\[
 \sfP\bl(S^A_K\ge(\sfE\tilX+\varepsilon/2)K\br)\le e^{-\tilde\alpha K}\qquad \text{for all $K$ large enough.}
\]
\end{proof}

\end{document}